\documentclass[a4paper,11pt]{book}
\usepackage[utf8]{inputenc}
\usepackage{lmodern}
\usepackage{amssymb,amsfonts,amsmath,amsthm,amscd}
\usepackage{amsmath,amsthm,amsfonts,amssymb,times,epsfig}
\usepackage{a4wide,times,amsmath,bbm,ifthen,graphicx,eurosym, color}
\usepackage{hyperref}
\usepackage{graphicx}
\usepackage[english]{babel}

\makeatletter
\renewcommand{\@chapapp}{}

\makeatother

\newcommand{\beq}{\begin{equation}} 
\newcommand{\eeq}{\end{equation}} 
\newcommand{\bea}{\begin{aligned}}
\newcommand{\eea}{\end{aligned}}
\newcommand{\bdm}{\begin{displaymath}}
\newcommand{\edm}{\end{displaymath}}
\newcommand{\barr}{\begin{array}}
\newcommand{\earr}{\end{array}}
\newcommand{\ben}{\begin{enumerate}}
\newcommand{\een}{\end{enumerate}}
\newcommand{\bde}{\begin{description}}
\newcommand{\ede}{\end{description}}

\newtheorem{teor}{Theorem}
\newtheorem{prop}[teor]{Fact}

\newtheorem{fact}{Fact}

\newtheorem{Def}[teor]{Definition}

\newtheorem{rem}[teor]{Remark}
\newcommand{\R}{\mathbb{R}}

\newcommand{\N}{\mathbb{N}}
\newcommand{\Z}{\mathbb{Z}}
\newcommand{\PP}{\mathbb{P}}

\newcommand{\defi}{\equiv} 

\newcommand{\be}{\beta}

\newcommand{\de}{\delta}

\newcommand{\la}{\lambda}

\newcommand{\s}{\sigma}
\newcommand{\vare}{\varepsilon}

\newcommand{\1}{\mathbbm{1}}

\title{\huge \textbf{Derrida's random energy models.}  \\ \Large From spin glasses to the extremes \\ of correlated random fields. \date{} }
\author{\textsc{Nicola Kistler} \\ Frankfurt University}

\begin{document}

\frontmatter
\maketitle


\tableofcontents


%
%
%

\chapter*{Introduction} These are notes for a minicourse I gave in Marseille in the spring of 2013, while holding the 
Jean Morlet Chair at the CIRM in Luminy. The goal/hope is to convey a {\it point of view} which captures some fundamental aspects shared by a variety of problems involving the extremes of large combinatorial structures. The formulation of this point of view in terms of an abstract theory (if there is any) eludes me, so I can only proceed by means of example. The level of rigor, as well as the mathematical infrastructure are intentionally low lest the unavoidable, model-related technicalities 
burdens the exposition. The emphasis is on the underlying picture, which is simple. 

The first chapter is devoted to a cursory discussion of some models introduced by Bernard Derrida in the context of mean field spin glasses in the 1980's, the random energy model and its generalization, REM and GREM respectively. These are simple Gaussian fields with a built-in hierarchical structure with finitely many levels, a feature which can be fully exploited in their rigorous treatment; the main steps of an approach based on elementary tools are briefly described. 

The approach becomes however cumbersome, to say the least, in case the assumption of finitely many hierarchies is not met, such is the case of the paradigmatic (Gaussian) hierarchical field, also known as the directed polymer on Cayley trees. For the latter, the main steps of a {\it multiscale refinement of the 2nd moment method} are worked out in the second chapter. The method seems to be new, although it is really just a neat reformulation of a well known tool. It can be applied with little modifications to a number of models which are neither Gaussian, such as certain issues of percolation in high dimensions, nor exactly hierarchical, such as the 2-dim Gaussian free field. More recently, the method has played a fundamental role in the study of two-dimensional cover times (a setting which is neither Gaussian, nor exactly hierarchical). This is briefly discussed in the third chapter, where the main steps of a general recipe are also laid out. 
The steps behind the multiscale refinement of the 2nd moment method are elementary, but  they conceptually rest on the deep insights provided by the work of theoretical physicists (Parisi, Derrida, to name a few) in spin glasses,  which one can vaguely summarize as follows: correlations play a role only at microscopic scales. The refinement also implements the related idea that hierarchical fields can be used as an approximation tool. Only the level of approximation matters: the greater it gets, the larger the number of scales one has to introduce. 
There is some reason to believe that this is a good way to address models falling in the class of the random energy model, i.e. with the simplest non-trivial freezing transition. (Despite the simplicity of the freezing transition, rather sophisticated models are known, or conjectured, to belong to the REM-class, such is the case for the extremes of the Riemann zeta-function along the critical line.)   

Many models in the REM-class are (approximately) self-similar across scales.  This is  a crucial property, and seemingly the reason for the extremes to behave essentially as in the independent setting. {\it ''What is particularly fascinating about the situation is that the underlying physics is either scaleless or has a single scale, but the mathematical machinery uses these multiple scales.''} The quote is taken from Simon's description of the Fr\"ohlich-Spencer multiscale analysis for the problem of Anderson localization \cite{simon}, but one may as well use it as a definition of the REM-class. In order to apply the multiscale refinement of the second moment method one thus 
first needs to {\it identify} the scales. In models with evident geometries, this step is typically straightforward.  However, the situation becomes quickly challenging if no 
geometry can be easily visualized, such is the case for certain number-theoretical issues, or questions related to random matrices. The last section touches upon a procedure of {\it local projections} which allows, in a number of cases, to construct scales from first principles. 

These notes are mostly about the level of the maximum of random fields. The situation is more involved if one is interested in the {\it microscopic} properties of the systems. An effective way to discuss this goes through the extremal process. This aspect is only marginally touched. It seems that models in the REM-class "interpolate" between the classical Poisson point process (REM case), and the Poisson {\it cluster} processes first emerged in the
analysis of  branching Brownian motion, BBM for short. (It is to be expected that the law of the clusters 
is model-dependent: at this level of precision it is unreasonable to hope for universality results.) In a way which can be made fully precise, BBM lies at the boundary of the REM-class, and models with more correlations automatically fall out of this class. In this case, however, little/nothing is known. 

In a nutshell, the point of view I am trying to marshal is that Derrida's random energy models are truly fundamental objects which play an important role also {\it beyond} their original context of mean field spin glasses. They are arguably the simplest models involving the concept of {\it scales}, so it should hardly come as a surprise that methods and insights gained in the study of the REMs are useful
whenever a multiscale analysis is needed. \\

{\bf Acknowledgments.} I am indebted to Erwin Bolthausen, who taught me all I know about the random energy models. It is a pleasure to thank Louis-Pierre Arguin, David Belius, Anton Bovier, Yan V. Fyodorov and Markus Petermann for the countless discussions on the topics of these notes. This work has been supported by the German Research Council in the SFB 611, the Hausdorff Center for Mathematics in Bonn, and Aix Marseille University/CIRM in Luminy through the Chair Jean Morlet. Hospitality of the University of Montreal where part of this work was done is also gratefully acknowledged.

\mainmatter

\chapter{Derrida's Random Energy Models}
The {\it random energy model}, the REM, and its generalization,  the GREM, have been introduced by Derrida \cite{derrida_rem, derrida_grem} in the 80's in order to shed some light on the mysteries of the Parisi theory for mean field spin glasses; within this context, 
the REMs have provided invaluable inputs ever since. 
On the other hand, recent remarkable advances in different fields show that the importance of the REMs goes well beyond the domain of spin glasses. In particular, methods and insights developed in the study of the REMs have been recently succesfully applied to the study of, e.g.
\begin{itemize} 
\item branching diffusions (see the lecture notes of Bovier \cite{anton}, and the review by Gou\'{e}r\'{e} \cite{gouere}, and references therein),
\item 2-dim Gaussian free field and more generally log-correlated random fields (see the lecture notes by Zeitouni \cite{zeitouni} and references therein),
\item  random matrices and number theory (Fyodorov, Hiary, and Keating \cite{fyodorov_keating_hiary}),
\item cover times (see e.g. \cite{dprz, beliusk}, and \cite{zeitouni} for a more exhaustive list of references)
\end{itemize}

Although the situation is only partially understood, the REM seems to be the foremost representative of a {\it universality class}. In spin glass terminology one may call this the 1-step replica symmetry breaking class, or 1RSB class for short.  Due to the simplicity of the model, the previous sentence might raise skepticism: why should complex models such as  those listed above fall into the universality class of the REM?  A possible answer simply restates Simon's quote: in many models, correlations are {\it strong} enough for the mathematical treatment to require a multiscale analysis, but {\it weak} enough so that the ensuing "physics" is single scale, as in the case of the REM.

\section{The REM}
Consider  a probability space $(\Omega, \mathcal F, \PP)$ and $N\in \N$. 

\begin{Def} The REM is a Gaussian random field 
$\left\{X_{\s}, \s= 1, \dots, 2^N\right\}$ where the $X's$ are independent, identically distributed centered Gaussian random variables with variance $N$. 
\end{Def}

The labels $\s$ will be referred to as {\it configurations}, whereas the values $X_\s$ as the {\it energies} of the configuration. \\

In these notes we are first and foremost concerned with the asymptotics of the extremes of random fields, in particular, the maximum 
\beq 
M_N \defi \max_{\s=1, \dots, 2^N} X_\s, \quad \text{as} \; N \to \infty.
\eeq

The behavior of the maximum of a random field is a question of fundamental importance in probability.  In the case of the REM the issue is particularly simple, due to the independence of the underlying random variables. With models in mind where this independence is no longer available, one  
seeks a flexible approach. For this, the so-called second moment method is not only  particularly efficient, but also one of the very few (general) tools available. Here are the main steps  in a cursory way. The reader interested in details is referred to the Oberwolfach Lecture Notes of Bolthausen \cite{bolthausen_sznitman}.

\subsection{The second moment method for the leading order of the maximum} \label{rem machinery}
Some notation used throughout the notes: $\cong$ stands for {\it equality on exponential scale}, namely $a_N\cong b_N$ if $a_N e^{-\epsilon N} \leq b_N \leq a_N e^{\epsilon N}$ for large enough $N$ and small enough $\epsilon >0$. Analogously, $\lesssim$ will denote {\it inequality on exponential scale}. Finally, $a_N \sim b_N$ if $a_N/b_N$ tends to a positive constant in the limit $N\to \infty$.\\

Let $\la \in \R$ and consider the counting random variable
\beq
\mathcal N_N(\la) \defi \sharp \left\{ \s = 1,\dots, 2^N: \; X_\s \geq \la N \right\}.
\eeq
The idea of the second moment method is to compare mean and variance of this counting random variable. The mean is easily derived thanks to standard large deviations estimates for Gaussians:  
\beq \label{one}
\mathbb{E} \mathcal N_N(\la) = 2^N \PP\left[ X_\s \geq \la N \right] \cong 2^{N} \exp\left( - \frac{\la^2}{2} N\right) = \exp N \left( \log 2 -\frac{\la^2}{2} \right)
\eeq

From \eqref{one}, and Borel-Cantelli, we immediately derive that $\mathcal N_N(\la) = 0$, almost surely for large enough $N$ {\it as soon as} $\la >\be_c \defi \sqrt{2\log 2}$. (The value $\be_c$ will constantly show up throughout the notes.)
In other words, $\PP$-almost surely,
\beq 
\frac{1}{N} \log \mathcal N_N(\la) \to -\infty, \quad \text{for}\; \la > \be_c.
\eeq

We next address the second moment of the counting random variable. Using the independence of $X_\s$ and $X_{\s'}$ for  $\s\neq \s'$, we get
\beq \bea
\mathbb{E} \mathcal N_N(\la)^2 & = \sum_{\s, \s'} \PP\left[ X_\s \cong \la N, X_{\s'} \geq \la N \right] \\
& = \sum_{\s=1}^{2^N} \PP\left[ X_\s \geq \la N\right] + \sum_{\s \neq \s'} \PP\left[ X_\s \geq \la N\right]^2 \\
& \cong 2^N \exp\left( - \frac{\la^2}{2} N\right) + 2^N(2^N -1)  \exp\left( - \la^2 N\right) .
\eea \eeq
Combining with \eqref{one}, we immediately deduce that 
\beq 
\frac{\mathbb{E} \mathcal N_N(\la)^2 - \mathbb{E}\left[\mathcal N_N(\la)\right]^2}{\mathbb{E}\left[\mathcal N_N(\la)\right]^2} \leq  \exp N \left( \frac{\la^2}{2}-\log 2\right),
\eeq
and this is vanishing {\it exponentially fast} as soon as $\la < \be_c$; a straightforward application of the Chebyshev inequality then implies the "quenched = annealed property", namely that one can bring the expectation inside of the logarithm: 
\beq \label{two}
\lim_{N\to \infty} \frac{1}{N} \log \mathcal N_N(\la) =  \lim_{N\to \infty} \frac{1}{N} \log \mathbb{E} \mathcal N_N(\la) = \log 2 - \frac{\la^2}{2} 
\eeq
almost surely, as soon as $\la < \be_c$. 

Summarizing, we have obtained the following almost sure statement:
\beq 
\lim_{N\to \infty} \frac{1}{N} \log \mathcal N_N(\la) = \begin{cases}
\log 2 - \frac{\la^2}{2} & \text{if}\; \la \leq \be_c\\
- \infty & \text{otherwise}.
\end{cases}
\eeq
(The case $\la = \be_c$ can be obtained by means of continuity arguments.) In other words, in the limit of large $N$, no configurations $\s$ will be found such that $X_\s > \be_c N$, whereas there are exponentially many configurations reaching heights which are less than $\be_c N$. Naturally, this threshold is our candidate for the {\it leading order} of the maximum of the random field, i.e. we expect 
\beq 
\max_\s X_\s =  \be_c N + o(N), \quad N\to \infty.
\eeq
with overwhelming probability. 

\subsection{The subleading order of the maximum - matching} \label{matching}
Once the leading order of the level of the maximum has been identified, one may be interested in 
the lower order corrections, and the finer properties of the system. In case of the REM this is of course well known from extreme value theory, but I present here an approach which allows to make educated guesses also in less standard situations. I will refer to this as the {\it method of matching}.

Let $a_N = \be_c N + \omega_N$ where $\omega_N = o(N)$ as $N\to \infty$, and consider the {\it extremal process}
\beq 
\Xi_N \defi \sum_{\s = 1}^{2^N} \de_{X_\s - a_N}
\eeq
This a random Radon measure which counts how many points of the collection 
$\{X_\s-a_N\}_\s$ fall into a given subset, i.e. for any compact 
$A \subset \R$ and $n\in \N$,
\beq 
\Xi_N(A) = n \Leftrightarrow \sharp\{ \s = 1 \dots 2^N: X_\sigma \in A + a_N \} = n.
\eeq
Now, if $a_N$ is the level of the maximum,  we should {\it find} configurations at these heights. 
In other words, it seems reasonable to require that 
$\Xi_N(A)$ is of order one in the large $N$-limit. Since this is hardly tractable from a quantitative point of view, we shall require that the {\it mean}, i.e. $\mathbb{E}\left[ \Xi_N(A) \right]$, remains of order one. This can be achieved by "matching constants". In fact,  
\beq \label{find}
\mathbb{E} \left[\Xi_N(A)\right] = 2^N \int_{A} \exp\left( - \frac{(x+a_N)^2}{2N}\right) \frac{dx}{\sqrt{2\pi N}}.
\eeq
Expanding the square using that $a_N = \be_c N + o(N)$ we have, asymptotically and for $x$ in compacts, 
\beq 
\exp\left( - \frac{(x+a_N)^2}{2N}\right) \sim 2^{-N} \exp\left( -\be_c \omega_N -\be_c x \right) \quad (N\to \infty).
\eeq
Since the term $2^{-N}$ matches the exponential factor in \eqref{find}, the latter remains asymptotically of order one provided $\omega_N$ compensates the $\sqrt{N}$-term appearing in the denominator of the Gaussian density, i.e. we shall require that 
\beq
\frac{1}{\sqrt{ N}} \exp\left( -\omega_N \be_c\right) \sim 1, \quad N\to \infty,
\eeq
which implies
\beq 
\omega_N = -\frac{1}{2\be_c} \log N.
\eeq
This turns out to be the correct choice. 

\begin{fact} Consider the extremal process of the REM, 
\beq
\Xi_N = \sum_\s \de_{X_\s -a_N}, \qquad  \label{max_rem}
a_N \defi \be_c N - \frac{1}{2 \be_c} \log N
\eeq
Then $\Xi_N$ converges weakly to a Poisson Point process $\Xi \defi (\xi_i)_{i\in \N}$ with intensity measure $\kappa e^{-\be_c x}dx$ for some $\kappa>0$ a numerical constant. In particular, the law of the maximum of the REM converges weakly, upon recentering, to a Gumbel distribution.
\end{fact}
\begin{proof}
This is of course well known. The proof is short and instructive, so we may as well give it here. In order to prove weak convergence, we may consider the convergence of Laplace functionals. Let therefore $\phi: \R \to \R_+$ be Borel-measurable, with compact support. It holds 
\beq \bea \label{laplace}
\mathbb{E}\left[ \exp\left( -\sum_\s \phi\left( X_\s - a_N\right)\right) \right] & = 
\mathbb{E}\left[ \exp\left( - \phi\left( X_1 - a_N\right)\right) \right]^{2^N}.
\eea \eeq
The trick (which ultimately rests on the fact that
$\phi(X_1-a_N)\equiv 0$ if $X_1-a_N$ falls out of the compact support of $\phi$) consists now in writing 
\beq 
\mathbb{E}\left[ e^{ - \phi} \right] = 1 - \mathbb{E}\left[ \left\{ 1- e^{-\phi}  \right\}\right].
\eeq
Using this, \eqref{laplace} equals 
\beq \label{laplace_two}
\left[ 1 - \int \left\{ 1- e^{-\phi(x)}\right\}   \exp\left( - \frac{(x+a_N)^2}{2N} \right) \frac{dx}{\sqrt{2\pi N}} \right]^{2^N}
\eeq
But thanks to the choice of $a_N$, for the Gaussian integral we have that
\beq \bea
& \int \left\{ 1- e^{-\phi(x)}\right\}   \exp\left( - \frac{(x+a_N)^2}{2N} \frac{dx}{\sqrt{2\pi N}} \right) \sim 2^{-N} \int \left\{ 1- e^{-\phi(x)}\right\} e^{-\be_c x}dx .
\eea \eeq
Plugging this in \eqref{laplace_two} and taking the limit $N\to \infty$ we indeed recover the Laplace 
transform of a Poisson point process.
\end{proof}

\section{The GREM}
In \cite{derrida_grem}, Derrida introduced the {\it generalized} random energy model, GREM for short. This is a Gaussian field with a built-in hierarchical structure. One can consider models with 
an arbitrary number of hierarchies, but conceptually one doesn't gain much, so we will first discuss the case with two levels. 

\begin{Def}  In the GREM with two levels, GREM(2) for short, the configurations are indexed by two-dimensional vectors 
\beq 
\s = (\s_1, \s_2)
\eeq
where without loss of generality we assume that $\s_1 = 1,\dots, 2^{N/2}$ and $\s_2 = 1,\dots, 2^{N/2}$. To each configuration we attach a random variable according to 
\beq \label{defi_grem}
X_{\s} \defi X_{\s_1}^{(1)}  + X_{\s_1, \s_2}^{(2)}
\eeq
where the $X_{\s_1}^{(1)}, \s_1 = 1 \dots 2^{N/2}$ are independent (centered) Gaussians of variance $a_1 N$, and, for each $\s_1$, the $X_{\s_1, \s_2}, \s_2 = 1\dots 2^{N/2}$ are independent (centered) Gaussians of variance $a_2 N$. As a normalization we choose $a_1, a_2$ such that $a_1+a_2 = 1$.
The variables at each level, $X^{(1)}_\cdot$ and $X^{(2)}_\cdot$, are assumed to be independent. 
\end{Def}

The GREM with two levels is a {\it correlated} random field (contrary to the REM): the covariance is given by
\beq
q_N(\s, \tau) \defi \mathbb{E} X_{\s} X_{\tau} = \begin{cases}
0 & \text{if}\; \s_1 \neq \tau_1 \\
a_1 N & \text{if}\; \s_1 = \tau_1 \, \text{but}\, \s_2 \neq \tau_2 \\
N & \text{otherwise}.
\end{cases} 
\eeq
In spin glass terminology, $q_N(\s, \tau)/N$ is the {\it overlap} of the configurations $\s$ 
and $\tau$. 

Under the light of (say) Slepian's Lemma we may expect correlations to have an impact 
on the level of the maximum, and in general, on the "geometry of extremes". In particular, since 
we are dealing with positive correlations, we may expect the maximum of the GREM to lie lower than
in the independent setting.  This is not always the case: as we will see, for certain choices of
$a_1, a_2$ the field still manages to reach levels as high as in the REM-case. The goal is to understand this qualitatively, i.e. to understand {\it how} the field manages to overcome correlations: there is good reason to believe that certain aspects of these strategies are "universal".

\subsection{Two scenarios for the leading order of the maximum}

We are interested in the extremes of the random field defined in \eqref{defi_grem}, in particular, the maximum
\beq 
M_N \defi \max_{\s = (\s_1, \s_2)} X_\s \,. 
\eeq
A natural way to address the asymptotics of $M_N$ is to follow the approach in the REM, i.e. to study the asymptotics of the counting random variable 
\beq \label{naive}
\mathcal N_N(\la) \defi \sharp\left\{ \s = (\s_1, \s_2): X_\s \geq \la N \right\}
\eeq
for $\la \in \R$. Just as in the REM one sees that for $\la > \be_c (=\sqrt{2\log 2})$, $\mathcal N_N(\la) = 0$ eventually for large $N$. However, following the route for the second moment one soon realizes that 
\beq
\text{var}\left[\mathcal N_N(\la)\right] \ll \mathbb{E}[\mathcal N_N(\la)]^2 
\eeq
only for values $\la$ which are much smaller than $\be_c$. In other words, second moment method fails. The reason for this is however easily identified: by considering \eqref{naive} we are completely forgetting the underlying tree-like structure, and the ensuing correlations. 
A better idea is to keep track of the underlying hierarchies. This can be achieved by considering
\beq
\mathcal N_N(\la_1, \la_2) \defi \sharp \left\{ \s = (\s_1, \s_2): X_{\s_1}^{(1)} \geq \la_1 N, 
X_{\s_1, \s_2}^{(2)} \geq \la_2 N \right\},
\eeq
for $\la_1, \la_2 \in \R$. 

By linearity of the expectation and the independence of first and second level we have
\beq 
\mathbb{E} N_N(\la_1, \la_2)  \cong 2^N \exp\left( -\frac{\la_1^2}{2 a_1} N - \frac{\la_2^2}{2a_2} N\right),
\eeq
from which we easily deduce that, with overwhelming probability and for large enough $N$,
\beq \label{all}
\mathcal N_N(\la_1, \la_2) \equiv 0 \quad \text{if}\quad \frac{\la_1^2}{2 a_1} + \frac{\la_2^2}{2a_2}  > \log 2. 
\eeq
There is yet another region where $\mathcal N_N$ vanishes almost surely, and which the naive approach through \eqref{naive} misses completely. To see this, let
\beq
\mathcal N_N(\la_1) \defi \sharp \left\{ \s_1= 1 \dots 2^{N/2}: X_{\s_1}^{(1)} \geq \la_1 N, 
 \right\}
\eeq
We first observe that, almost surely, 
\beq 
\mathcal N_N(\la_1) = 0 \Rightarrow \mathcal N_N(\la_1, \la_2) = 0\,.
\eeq
But 
\beq 
\mathbb{E} \mathcal N_N(\la_1)  \cong 2^{N/2} \exp\left( -\frac{\la_1^2}{2 a_2} N \right)
\eeq
hence, by Borel-Cantelli, 
\beq \label{firstlevel}
\mathcal N_N(\la_1) = 0 \quad \text{for} \quad \frac{\la_1^2}{2a_1} > \frac{1}{2}\log 2.
\eeq
Summarizing, 
\beq \bea \label{conditions}
& \lim_{N\to \infty} \frac{1}{N} \log \mathcal N_N(\la_1, \la_2) = -\infty\\
& \hspace{2cm} \text{as soon as} \quad \frac{\la_1^2}{2a_1} > \frac{1}{2}\log 2 \quad \text{or} \quad \frac{\la_1^2}{2 a_1} + \frac{\la_2^2}{2 a_2}  > \log 2
\eea \eeq
The second moment of the counting random variable can be easily controlled by rearranging the 
ensuing sum according to the possible correlations: 
\beq \bea
\mathbb{E}\left[ \mathcal N_N(\la_1, \la_2)^2 \right] & = \sum_{\s, \tau}
\PP\left[ X_{\s_1}^{(1)} \geq \la_1 N, 
X_{\s_1, \s_2}^{(2)} \geq \la_2 N, X_{\tau_1}^{(1)} \geq \la_1 N, 
X_{\tau_1, \tau_2}^{(2)} \geq \la_2 N\right] \\
& = \sum_{\s = \tau} + \sum_{\s_1 = \tau_1, \s_2 \neq \tau_2 } +  \sum_{\s_1 \neq \tau_1} \\
& \lesssim  \mathbb{E} \left[ \mathcal N_N(\la_1, \la_2)\right] + 2^{3 N/2} \exp\left( - \frac{\la_1^2}{2 a_1} N - 
\frac{\la_2^2}{a_2} N\right)  + \mathbb{E} \left[ \mathcal N_N(\la_1, \la_2)\right]^2 \,. 
\eea \eeq
From this one deduces that 
\beq 
\frac{\text{var}\left[ \mathcal N_N(\la_1, \la_2) \right]}{
\mathbb{E} \left[ \mathcal N_N(\la_1, \la_2)\right]^2} \lesssim 
2^{-N} \exp\left( \frac{\la_1^2}{2 a_1} N + \frac{\la_2^2}{2a_2} N\right) + 2^{-N/2} \exp\left( \frac{\la_1^2}{2 a_1} N \right)
\eeq
The crucial point is that both expressions can be made exponentially small on the {\it complement} of the $(\la_1, \la_2)$-set appearing in \eqref{conditions}. In other words, 
\beq 
\frac{\text{var}\left[ \mathcal N_N(\la_1, \la_2) \right]}{
\mathbb{E} \left[ \mathcal N_N(\la_1, \la_2)\right]^2} \to 0 \quad (N\to \infty)
\eeq
exponentially fast as soon as 
\beq  \label{qa}
\frac{\la_1^2}{2 a_1} < \frac{1}{2}\log 2 \quad \text{and} \quad \frac{\la_1^2}{2 a_1} + \frac{\la_2^2}{2 a_2} < \log 2.
\eeq
By Chebyshev's inequality, this implies the "quenched = annealed" statement
\beq 
\lim_{N\to \infty} \frac{1}{N} \log \mathcal N_N(\la_1, \la_2) =\lim_{N\to \infty} \frac{1}{N} \log \mathbb{E} \mathcal N_N(\la_1, \la_2)  = \log 2 - \frac{\la_1^2}{2a_1} -\frac{\la_2^2}{2a_2},
\eeq
provided $\la_1, \la_2$ satisfy the conditions \eqref{qa}. Putting all together, 
we have obtained the following: 
\beq \label{ally}
\lim_{N\to \infty} \frac{1}{N} \log \mathcal N_N(\la_1, \la_2) 
= \begin{cases}
\log 2 - \frac{\la_1^2}{2a_1} -\frac{\la_2^2}{2a_2}, & \text{if}\; \frac{\la_1^2}{2 a_1} \leq \frac{1}{2}\log 2, \frac{\la_1^2}{2 a_1} + \frac{\la_2^2}{2a_2} \leq \log 2 \\
-\infty & \text{otherwise}.
\end{cases}
\eeq
(Equality in the side constraints can be obtained by continuity arguments.) 
In particular we see that no configurations is found  s.t. $X_\s \geq (\la_1 + \la_2) N$ as 
soon as one of the side-constraints in \eqref{ally} is violated, whereas there are exponentially 
many $\s's$ reaching heights $(\la_1 + \la_2) N$ if the side-constraints are satisfied strictly: the threshold separating these two scenarios is our natural guess for the level of the maximum. In other words, we expect that 
\beq 
M_N = m N + o(N)
\eeq
with overwhelming probability and for $m = m(a_1, a_2)$ the solution of the optimization problem 
\beq
m \defi \sup_{(\la_1, \la_2) \in \R^2} \left\{ \la_1 + \la_2: \frac{\la_1^2}{2a_1} \leq \frac{1}{2}\log 2, \frac{\la_1^2}{2 a_1} + \frac{\la_2^2}{2 a_2} \leq \log 2 \right\}
\eeq
This is a two dimensional convex optimization problem. It can be solved through Lagrange multipliers. The upshot of the elementary calculations is as follows. There are  two different scenarios, depending on the choice of $a_1, a_2$ : \\

 {\bf Case 1:} $a_1 \leq a_2$. In this case the supremum is achieved in 
\beq \label{caseone}
\la_1^{\text{max}} = \sqrt{ 2 \log 2} a_1 , \; \la_2^{\text{max}} = \sqrt{2\log 2} a_2,
\eeq
Recalling the normalization $a_1+a_2 = 1$, we would thus obtain that
\beq
\max_\s X_\s = \be_c N + o(N),
\eeq
(recall that $\be_c= \sqrt{2\log 2}$)
with overwhelming probability, {\it exactly} as in the case of the random energy model. So, our considerations suggest that correlations do not matter as long as $a_1 \leq a_2$.\\

 {\bf Case 2:} $a_1 >a _2$. In this case, the optimal $\la's$ saturate the side constraints, 
\beq  \label{casetwo}
\la_1^{\text{max}} = \sqrt{a_1 \log 2}, \; \la_2^{\text{max}} = \sqrt{a_2 \log 2},
\eeq
hence
\beq \label{strict less}
\max_\s X_\s =  (\sqrt{a_1 \log 2}+ \sqrt{a_2 \log 2})N + o(N),
\eeq
with overwhelming probability. Elementary convexity arguments show the leading order of the GREM is {\it strictly smaller} than $\be_c$,  the value for the leading order of the REM. 
In other words, our considerations suggest that correlations do matter. \\

The above guesses for the leading order of the maximum are indeed correct. I will not give a proof of this but refer the interested reader to \cite{bovier_kurkova} for details.

\subsection{Three scenarios for the subleading order of the maximum}
Since the analysis at the level of large deviations for the leading order of the maximum of the GREM leads to two 
different scenarios, one is perhaps tempted to believe that the same is true for the subleading orders, 
and the extremal process. This is not quite correct: as far as the finer properties 
of the system are concerned, one has to distinguish between {\it three} different scenarios: $a_1 < a_2$ {\it strictly}, $a_1= a_2$, and $a_1 > a_2$.
The next goal is to discuss {\it qualitatively} the different physical pictures, and the strategies adopted by the random field to achieve the extreme values. 

\subsubsection{The case $a_1>a_2$}
In this situation, the way the GREM achieves the maximal values is very natural: one looks for those $\s_1$ which maximizes the contribution on the first level of the tree. But the first level is nothing but a REM,  and one can apply the machinery from Section \ref{rem machinery} to see that $\max_{\s_1} X_{\s_1}^{(1)} \approx \sqrt{\log a_1} N$. Let us assume without loss of generality that this maximum is achieved in $\s_1 =1$. The optimal strategy is then to search for the extremal $\s_2$-configuration in the tree attached to such $\s_1= 1$; this tree being again 
a REM, one easily sees that $\max_{\s_2} X_{1, \s_2}^{(2)} \approx \sqrt{\log a_2} N$. This is consistent   with \eqref{strict less}. Under this light, it is easy to guess what the level of the maximum  should be (lower orders included): it should be given by the sum of the two corresponding maxima. More precisely, let 
\beq 
a_N^{(1)} \defi \sqrt{ a_1 \log 2 } N - \frac{a_1}{2 \sqrt{a_1 \log 2}} \log N
\eeq
and 
\beq 
a_N^{(2)} \defi \sqrt{ a_2 \log 2 } N - \frac{a_2}{2 \sqrt{a_2 \log 2}} \log N
\eeq
Remark that $a_N^{(1)}$ is the level of the maximum of the REM associated to the first level, 
whereas $a_N^{(2)}$ is the level of the maximum of any of the REM associated to the second level. 
One can then prove that the point process associated to the first level, to wit
\beq \label{firstlev}
\left(X_{\s_1}^{(1)}- a_N^{(1)} \right)_{\s_1 = 1}^{2^{N/2}}
\eeq
converges weakly to a Poisson point process with intensity measure proportional (up to irrelevant numerical constant) to $e^{-t \sqrt{\log 2/a_1}} dt$. Furthermore, for {\it given} $\s_1$, the point process associated to the second level, to wit 
\beq \label{seclev}
\left(X_{\s_1, \s_2}^{(2)}- a_N^{(2)} \right)_{\s_2 = 1}^{2^{N/2}}
\eeq
converges weakly to a Poisson point process with intensity measure proportional (up to irrelevant numerical constant) to $e^{-t \sqrt{\log 2/a_2}} dt$. Let now 
\beq 
a_N \defi a_N^{(1)} + a_N^{(2)}
\eeq
From the weak convergence of the point processes \eqref{firstlev}, \eqref{seclev} one then easily derives that the {\it full} extremal process, to wit 
\beq
\Xi_N \defi \left(X_{\s} -a_N \right)_{\s = (\s_1, \s_2)}
\eeq
converges weakly to the {\it superposition} of the two limiting objects. To formulate this precisely, 
let $\left(\xi_{i_1}^{(1)},  i_1 \in \N\right)$ be  a Poisson point process with intensity measure $e^{-t \sqrt{\log 2/a_1} } dt$ (up to irrelevant numerical constant). For $i_1\in \N$, consider a Poisson point process $\left(\xi_{i_1,i_2}^{(2)}, i_2 \in \N \right)$ with density $e^{-t \sqrt{\log 2/a_2} } dt$. (All point processes are assumed to be independent.) Finally, let 
\beq \label{derrida_ruelle}
\Xi \defi \left( \xi_{i_1}^{(1)} + \xi_{i_1, i_2}^{(2)}, i_1, i_2 \in \N \right)
\eeq

\begin{fact} $\Xi_N \to \Xi$ weakly. 
\end{fact}

\begin{rem} To my knowledge, the point process in \eqref{derrida_ruelle} (and its generalization to an arbitrary number of levels) has made its first appearance in the landmark paper by 
Ruelle \cite{ruelle} on Derrida's REM and GREM.  Such point processes enjoy truly remarkable properties, and play a fundamental role in the Parisi theory of spin glasses. They are nowadays called Derrida-Ruelle 
cascades. 
\end{rem}

\subsubsection{The case $a_1 < a_2$,}
In this situation, it turns out that correlations are too weak to have an impact on the extremes, and 
the system "collapses" to a REM.  To see this, let 
\beq 
a_N \defi \be_c N - \frac{1}{2 \be_c} \log N.
\eeq
(This is the level of the maximum in the REM). Furthermore, let 
$\Xi \defi (\xi_i, i \in N)$ denote a Poisson point process with intensity measure 
$e^{-\be_c t} dt$ (up to irrelevant numerical constant) and let
\beq
\Xi_N \defi (X_{\s} - a_N, \s = (\s_1, \s_2) )
\eeq
be the extremal process.
\begin{fact}
$\Xi_N \to \Xi$, weakly.
\end{fact}
I will not give the (simple) proof of this statement, see e.g. \cite{bovier_kurkova_one, bolti_kis}. Still, the outcome is a bit surprising and deserves some comments. In fact, the GREM(2) is a correlated Gaussian field but in the case 
where $a_1 < a_2$, correlations are too weak to be detectable at the level of 
the extremal process. The reason for this is a certain entropy vs. energy competition, as can be seen 
through the following considerations: one can prove by a simple union bound and standard Gaussian estimates that for any given compact $A\subset \R$, and $a_N$ as in \eqref{max_rem}
\beq \bea
\lim_{N\to \infty} \PP\left[\exists \s, \tau : \s_1 = \tau_1, \s_2\neq \tau_2\; \text{and}\; X_\s -a_N \in A, X_\tau -a_N \in A \right] = 0. 
\eea \eeq
In other words, extremal configurations must  differ in the first index; by construction, this implies that the associated random variables are independent, and this justifies the onset of the Poisson point process. This also suggests that the way the correlated random field "overcomes correlations" is by looking for the maximum among all possible ("finite") subsets of configurations with different first indices,  a strategy which is fundamentally different from that of the GREM(2) with $a_1>a_2$.
The reader interested in the details is referred to \cite{bolti_kis}. 

\subsubsection{The case $a_1=a_2$} 
This is a somewhat "critical" case. Since the field still manages to overcome correlations reaching  
high values which are, to leading order, the same as in the REM-case, one is tempted to guess 
that similar mechanisms as in the case $a_1 < a_2$ are at play. However, this turns out to be incorrect: the strategy used by the random field to overcome correlations is more sophisticated.  
As a matter of fact, this critical case is in many aspects the most interesting, and technically most demanding among the possible scenarios. The reason for this is  the inherent {\it self-similarity} (the model is the superposition of two identical random energy models). Anticipating, one may say that the  self-similarity makes the model 
(and its generalization: the hierarchical field which is discussed in the next section) ubiquitous in fields beyond the context of spin glasses.

Here is first indication that things might not be as easy as they seem. Recall the normalization $a_1 + a_2 = 1$; at criticality, it thus follows that $a_1 = a_2 = 1/2$.
According to the LDP \eqref{caseone}, the optimal $\la's$ are given by 
\beq 
\la_1^{\text{max}} = \sqrt{2\log 2} a_1 = \frac{\be_c}{2} = \la_2^{\text{max}}
\eeq
On the other hand, we would get {\it exactly} the same values through the 
LDP leading to \eqref{casetwo}: both scenarios yield the same leading order for the level of the maximum, $\be_c$. The subleading corrections are however a different matter: in the first case (superposition of two REMs) we would get 
\beq 
\be_c N - \frac{1}{\be_c} \log N,
\eeq
whereas in the second case (collapsing to a REM)
\beq \label{riddle}
\be_c N - \frac{1}{2 \be_c} \log N.
\eeq
Remark that the logarithmic corrections differ by a factor of two. So, which is the correct level of the maximum? The answer to the riddle is \eqref{riddle}. The  point is however not so much the numerical value (it just {\it happens} to be the same as in the REM) but the physical mechanism which leads to this outcome. The starting point is again the LDP-analysis for the counting random variable $\mathcal N_N$. Specifying the outcome of these considerations in the critical case $a_1  = a_2 = 1/2$, extremal configurations satisfy 
\beq 
X_\s =\be_c N+ \omega_N, \quad \omega_N = o(N).
\eeq
The question is again how to choose $\omega_N$, and for this we would like to use the 
method of matching as discussed in Section \ref{matching}. However, we first make a fundamental observation: the LDP analysis of the counting random variable $\mathcal N_N$ tells us that no configuration $\s = (\s_1, \s_2)$ will be found (in the large $N$-limit) such that 
\beq 
X_{\s_1}^{(1)} >\frac{\be_c}{2} N ,
\eeq
see in particualr \eqref{firstlevel}. In other words, configurations contributing to the extremal process should also satisfy the condition 
\beq \label{condless}
X_{\s_1}^{(1)} \leq \frac{\be_c}{2} N \,.
\eeq
Summarizing, with $a_N \defi \be_c N+ \omega_N$, a "candidate" for the extremal process is the {\it thinned} point process
\beq \label{thinned}
\Xi_N^{(\leq)} \defi \sum_{\s}^{(\leq)} \de_{X_{\s} - a_N}
\eeq
where summation $\sum^{(\leq)}$ is only over those configurations which satisfy \eqref{condless}. Following the method of matching we thus require that for arbitrary compact $A\subset \R$
\beq \label{order_one_thinned}
\mathbb{E}\left[ \Xi_N^{(\leq)}(A) \right] \text{remains of order one in the limit}\; N\to \infty.
\eeq
But with $\s = (\s_1, \s_2)$ any reference configuration,
\beq \bea
\mathbb{E}\left[ \Xi_N^{(\leq)}(A) \right] & = 2^N \PP\left[X_{\s_1}^{(1)} \leq \frac{\be_c}{2} N , X_{\s} -a_N \in A \right] \\
& = 2^N \int_{A} \PP\left[X_{\s_1}^{(1)} \leq\frac{\be_c}{2} N  \Big| X_{\s} = x+ a_N \right] 
\exp\left( -\frac{(x+ a_N)^2}{2 N}\right)  \frac{dx}{\sqrt{2\pi N}}
\eea \eeq
We now expand the Gaussian density, 
\beq \bea
& \exp\left( -\frac{(x+ a_N)^2}{2 N}\right)  \frac{1}{\sqrt{2\pi N}}  \sim 2^{-N}  \exp\left(-\be_c x\right) \frac{
\exp{ \left[- \be_c\omega_N\right]}}{\sqrt{N}}.
\eea \eeq
Compared to the REM, one notices the additional term
\beq \label{brownianbridge}
\PP\left[X_{\s_1}^{(1)} \leq \frac{\be_c}{2} N  \Big| X_{\s} = x+ a_N \right].
\eeq
In the simple case of critical GREM(2), this term gives rise to yet another numerical constant. However, contrary to the many irrelevant numerical constants encountered along the way, this number encodes important information. In fact, in many applications (some will be discussed below) this term plays an absolutely fundamental, {\it structural} role, so it is important to understand what lies behind it. Remark 
that $X_\s = X_{\s_1}^{(1)}+ X_{\s_1, \s_2}^{(2)}$, where $X^{(1)}$ and $X^{(2)}$ are independent 
Gaussians of variance $N/2$. Therefore,  for given $\s = (\s_1, \s_2)$ we may see the process $\left\{0, X_{\s_1}^{(1)}, X_{\s} \right\}$ as the first steps of a random walk (issued at zero) with Gaussian increments, and \eqref{brownianbridge} is then the probability that a certain (discrete) Brownian bridge of lifespan $2$ stays below a certain threshold at time $1$. Here is a nice trick to compute this probability: the idea is to consider the random variable $X_{\s_1}^{(1)} - \frac{1}{2} X_\s$, which is a Gaussian with the property that
\beq 
\mathbb{E}\left[ X_{\s} \left( X_{\s_1}^{(1)} - \frac{1}{2} X_\s \right)\right] = 
\mathbb{E} X_{\s} X_{\s_1}^{(1)} - \frac{1}{2} \mathbb{E}  X_{\s} X_\s = 0.
\eeq
This implies that $X_\s$ and $X_{\s_1}^{(1)} - \frac{1}{2} X_\s$ are independent. Using this,
\beq \bea
& \PP\left[X_{\s_1}^{(1)} \leq \frac{\be_c}{2} N  \Big| X_{\s} = x+ a_N \right] \\
& \qquad = \PP\left[X_{\s_1}^{(1)} -\frac{1}{2} X_\s \leq \frac{\be_c}{2} N  - \frac{1}{2}X_\s \Big| X_{\s} = x+ a_N \right] \\
& \qquad = \PP\left[X_{\s_1}^{(1)} -\frac{1}{2} X_\s \leq \frac{\be_c}{2} N  - \frac{1}{2}(x+a_N) \Big| X_{\s} = x+ a_N \right]  \\
& \qquad = \PP\left[X_{\s_1}^{(1)} -\frac{1}{2} X_\s \leq\frac{\be_c}{2} N  - \frac{1}{2}(x+a_N) \right],
\eea \eeq
the last equality by independence. Using that $a_N = \be_c N + \omega_N$, we get that \eqref{brownianbridge} equals
\beq \label{lying low}
\PP\left[X_{\s_1}^{(1)} -\frac{1}{2} X_\s \leq - \frac{1}{2}(x+\omega_N) \right]
\eeq
Since we assume $\omega_N = O(\log N)$, and $X_{\s_1}^{(1)} -\frac{1}{2} X_\s$ is  a centered Gaussian with variance of order $N$, one checks that 
\beq \label{onehalf}
\eqref{brownianbridge}  \to \frac{1}{2} \qquad (N\to \infty). 
\eeq
Going back to \eqref{order_one_thinned}, we see that we should choose 
\beq
\omega_N = -\frac{1}{2 \be_c} \log N, 
\eeq
{\it just as in the case of the REM}. However, there are a number of twists. 
First, due to \eqref{onehalf}, the intensity measure of the limiting extremal process (provided the similarities go through) will be {\it half} that of the extremal process of the REM.  Second, the above discussion doesn't quite explain what happens at the physical level. A good way to understand what is at stake goes as follows. One easily checks that a Gaussian random variable of variance of order $O(N)$ which is required to stay below a straight line of "height" $o(\sqrt{N})$, such is the case in \eqref{lying low}, lies way lower, namely at heights of order $-\sqrt{N}$ (this is an instance of the entropic repulsion, a phenomenon which is well-known in the statistical mechanics of random surfaces). This turns out to be the strategy used by the field to overcome correlations: one can prove that for $\s$ an extremal configuration, i.e. such that 
\beq 
X_{\s}  \approx \be_c N - \frac{1}{2 \be_c} \log N,
\eeq
it holds that 
\beq \label{cons_start}
X_{\s_1}^{(1)} \approx \frac{\be_c}{2} N-\vare \sqrt{N},
\eeq
for some $\vare >0$, with overwhelming probability. In other words: extremal configurations are {\it never} extremal on the first level of the tree. (The same is true in the case $a_1< a_2$, but due to a less sophisticated mechanism). The reason why this is a good strategy has to do with the self-similarity of the critical GREM, reflected in the fact that $a_1 = a_2$. One can easily check that
\beq \label{many}
\sharp \left\{\s_1 = 1\dots 2^{N/2}:  X_{\s_1}^{(1)} \approx \frac{\be_c}{2} N -\vare \sqrt{N} \right\} \approx \exp (\vare \sqrt{N} \be_c),
\eeq
with overwhelming probability.  For a given $\s_1$ to be extremal it means that the associated $X_{\s_1, \s_2}^{(2)}$ must make up for the sub-optimality of the first level, i.e. it has to make an unusually large jump of order
\beq 
a_N - \left(  \frac{\be_c}{2} N -\vare \sqrt{N}  \right)  =\frac{\be_c}{2} N +\vare \sqrt{N}.
\eeq
But for given $\s_1$, straightforward estimates show that 
\beq \label{jump}
\PP\left[ \exists \s_2= 1 \dots 2^{N/2}: X_{\s_1, \s_2}^{(2)} \approx \frac{\be_c}{2} N +\vare \sqrt{N} \right] \approx \exp( -\vare\be_c \sqrt{N })
\eeq
Combining this with \eqref{many} we thus see that 
\beq
\PP\left[ \exists \s = (\s_1, \s_2):  X_{\s_1}^{(1)} \approx \frac{\be_c}{2} N -\vare \sqrt{N}, X_{\s} \approx a_N \right] \sim 1 \quad (N\to \infty).
\eeq
(Of course, the situation is slightly more complicated, since one also has to "integrate over $\vare$", but the above back-of-the-envelope computations give a good idea of what is going on). On the other hand one also checks that to given $\s_1$ the probability to find {\it two} configurations $\s_2, \tau_2$ within the {\it same} tree which make the unusually large jump is at most
\beq \bea \label{smallsmall}
\PP\left[ \exists \s_2 \neq \tau_2 : X_{\s_1, \s_2}^{(2)} \approx  \frac{\be_c}{2} N +\vare \sqrt{N}, X_{\s_1, \tau_2}^{(2)} \approx  \frac{\be_c}{2} N +\vare \sqrt{N} \right] \approx \exp( - 2\vare \be_c \sqrt{N } ),
\eea \eeq
i.e. twice as small as \eqref{jump}, and this implies 
\beq \bea \label{cons_finish}
& \PP\left[ \exists \s, \tau, \s_1= \tau_1:  X_{\s_1}^{(1)}  \approx \frac{\be_c}{2} N-\vare \sqrt{N}, X_{\s} \approx a_N , X_{\tau} \approx a_N\right] \approx \exp( -\vare \be_c\sqrt{N }) .
\eea \eeq
(This can be rigorously established by a simple union bound, and the usual Gaussan estimates.) This can be used to prove that the level of the maximum of the critical GREM$(2)$ is 
indeed $a_N = \be_c N - 1/(2\be_c) \log N$, just as in the REM case. For this, it is useful to "give 
oneself some room": for $\delta >0$, 
\beq
a_N^{(\de)} = \be_c N -\frac{1-\delta}{ 2\be_c} \log N. 
\eeq
Remark that 
\beq
a_N^{(-\de)} \leq a_N \leq a_N^{(\de)}.
\eeq
One direction, the upper bound, is easy: a simple Markov inequality (union bound) shows that for any
\beq
\lim_{N\to \infty} \PP\left[ \max_\s X_\s \geq a_N^{(\de)} \right] = 0. 
\eeq
As for a lower bound, it can be established by a {\it restricted Payley-Zygmund inequality}. We claim that for any $\delta >0$, it holds 
\beq \label{claim_lower}
\PP\left[ \max_\s X_\s \geq a_N^{(-\delta)}\right] \to 1, \qquad (N\to \infty).
\eeq
To see this, consider $\varepsilon >0$ and denote by $\max_{\s}^{\vare} X_\s$ the maximum of the field {\it restricted} to those configurations $\s$ with 
first level satisfying 
\beq
X_{\sigma_1}^{(1)} \leq \frac{\be_c}{2} N - \varepsilon \sqrt{N}
\eeq
Since the maximum over a set is larger than the maximum over a subset, we have: 
\beq \bea 
\PP\left[ \max_\s X_\s \geq a_N^{(-\delta)} \right] &\geq \PP\left[ \max^\vare_\s X_\s \geq a_N^{(-\delta)} \right] \\
& \geq \frac{\mathbb{E}\left[\sharp\{\s: X_\s \geq a_N^{(-\delta)}, X_{\sigma_1}^{(1)} \leq \frac{\be_c}{ 2} N - \varepsilon \sqrt{N}\}\right]^2}{\mathbb{E}\left[\sharp\{\s: X_\s \geq a_N^{(-\delta)}, X_{\sigma_1}^{(1)} \leq \frac{\be_c}{2} N - \varepsilon \sqrt{N}\}^2\right]}, \label{restr_pz}
\eea \eeq
the last line by Payley-Zygmund. Now,
\beq \bea \label{riscrivi}
& \mathbb{E}\left[\sharp\{\s: X_\s \geq a_N^{(-\delta)}, X_{\sigma_1}^{(1)} \leq \frac{\be_c}{2} N - \varepsilon \sqrt{N}\}^2\right] \\
& = \sum_{\s, \tau} \PP\left[ X_{\sigma_1}^{(1)}, X_{\tau_1}^{(1)} \leq \frac{\be_c}{2} N - \varepsilon \sqrt{N}, X_\s, X_\tau \geq a_N^{(-\delta)}
 \right] \\
& = \sum_{\s_1 \neq \tau_1} + \sum_{\s_1 = \tau_1} 
\eea \eeq
Now, we have the following simple facts: 
\begin{itemize}
\item if $\s_1 \neq \tau_1$ we have by construction that $X_\s$ and $X_\tau$ are independent. In other words, the first sum in the last line of \eqref{riscrivi} is nothing but (less than) the numerator in 
\eqref{restr_pz}. 
\item By considerations similar to those in \eqref{cons_finish}, 
the second sum in \eqref{riscrivi} is of order $\exp(-\varepsilon \be_c \sqrt{N})$.
\end{itemize}
Using these observations in \eqref{restr_pz} we thus have
\beq
\PP\left[ \max_\s X_\s \geq a_N^{(-\delta)} \right] \geq 
\frac{\mathbb{E}\left[\sharp\{\s: X_\s \geq a_N^{(-\delta)}, X_{\sigma_1}^{(1)} \leq \frac{\be_c}{2}N - \varepsilon \sqrt{N}\}\right]^2}{\mathbb{E}\left[\sharp\{\s: X_\s \geq a_N^{(-\delta)}, X_{\sigma_1}^{(1)} \leq \frac{\be_c}{2}N - \varepsilon \sqrt{N}\}\right]^2 + \exp(-\varepsilon \be_c \sqrt{N}) }
\eeq
A simple computation shows that 
\beq
\mathbb{E}\left[\sharp\{\s: X_\s \geq a_N^{(-\delta)}, X_{\sigma_1}^{(1)} \leq \frac{\be_c}{2}N - \varepsilon \sqrt{N}\}\right] \sim  N^{\de/2}, 
\eeq
hence 
\beq 
\PP\left[ \max_\s X_\s \geq a_N^{(-\delta)} \right] \geq \frac{N^{\de}}{N^{\de}+ \exp(-\varepsilon \be_c \sqrt{N}) } \to 1 \qquad (N\to \infty),
\eeq
proving the claim \eqref{claim_lower}. \\

The above analysis can be used to establish the convergence of the full extremal process: since extremal configurations must necessarily differ on the first level (the associated random variables are independent), one naturally expects the limiting process to be Poissonian. This is indeed correct. With 
$a_N = \be_c N - 1/(2\be_c) \log N$, and denoting by $\Xi_N \defi \sum_{\s} \de_{X_\s-a_N}$
the extremal process of the critical GREM(2), it holds: 
\begin{fact} \cite{bovier_kurkova_one, bolti_kis} 
$\Xi_N $ converges weakly to a Poisson point process with intensity measure "half the one of 
the REM".
\end{fact}

The reader notices the following twists: 
\begin{itemize}
\item first, the above statement concerns the full extremal process, and {\it not} the thinned version \eqref{thinned} which was instrumental for our considerations. Indeed, it takes some work to {\it prove} that configurations not satisfying \eqref{condless} do not contribute to the extremal process. I am not going to dwell on this at this point, since I will give some pointers towards 
a general recipe while discussing the subleading order of the hierarchical field, see Section \ref{smart markov} below. 
\item Second, the reduced intensity of the limiting process ("half the one of the REM") is of course due to \eqref{onehalf}.
\end{itemize}

\section{The critical GREM with $K>2$ levels}
The picture underlying the critical GREM(2) is stable, with minor adjustments needed to cover the GREM$(K)$ for generic $K\in \N$. 

\begin{Def} Consider $K$-dimensional vectors $\s = (\s_1, \s_2, \dots, \s_K)$; without loss of generality we assume 
$\s_j = 1\dots 2^{N/K}$ for $j=1\dots K$. The random field is then given by 
\beq 
X_\s = X_{\s_1}^{(1)} + X_{\s_1, \s_2}^{(2)}+ \dots + X_{\s_1 \dots \s_K}^{(K)}
\eeq
where all the random variables on the r.h.s. above are independent with $X^{(j)}_\cdot \sim 
\mathcal N(0, N/K)$. 
\end{Def}

Again, we are interested in the maximum, 
\beq 
M_N \defi \max_{\s = (\s_1, \dots, \s_k)} X_\s . 
\eeq
As a first step, we focus on the leading order, namely $m_k\in \R$ s.t. $M_N = m_k N + o(N)$
with overwhelming probability. The analysis goes along the same lines as in GREM(2), i.e. one considers 
\beq \label{klevels}
\mathcal N_N(\la_1, \dots, , \la_k) \defi \sharp \left\{ \s = (\s_1, \dots, \s_K): 
X_{\s_1}^{(1)} \geq \la_1 N, \dots, X_{\s_1 \dots \s_k}^{(K)} \geq \la_K N
\right\}.
\eeq
Second moment method then leads to the following variational principle for the leading order of the maximum:
\beq \label{finitely_cons}
m_K \defi \sup\left\{ \la_1 + \dots + \la_k: \sum_{j=1}^L \frac{\la_j^2}{2 a_j} \leq \frac{L}{K} \log 2, \text{for}\; L =1 \dots K \right\}
\eeq
The maximizers are easily seen to be given by
$ \la_j^{\text{max}} = \be_c/K$, hence
$m_K = \be_c,$
{\it exactly} as in the REM. The same is true for the level of the maximum (lower orders included) which turns out to be
\beq 
a_N = \be_c N - \frac{1}{2\be_c} \log N.
\eeq
The only quantitative difference is that the extremal process converges towards a Poisson point process with density the one of the extremal process associated to the REM {\it reduced} by a factor $C_K$, where $C_K$ is the probability that a Brownian bridge of lifespan $K$ stays below zero at the times $1, 2, \dots, K-1$. In other words, with $\mathfrak Z_K$ such a Brownian bridge: 
\beq \label{bbridge}
C_K = \PP\left[ \mathfrak Z_k\left(1\right) \leq 0, \mathfrak Z_k\left(2\right)\leq 0, \dots, \mathfrak Z_K\left(K-1\right)\leq 0\right].
\eeq
Qualitatively, the picture underlying the extremes is exactly as in the critical GREM(2). It turns out that 
configurations contributing to the extremal process 
are {\it not} extremal at intermediate levels, but lie lower (by order $\sqrt{N}$) than the relative 
maximum: sligthly more precisely, a {\it necessary} requirement for $X_{\s} \approx \be_c N$ to hold is that 
\beq \label{entropic conditions}
\sum_{j=1}^L X_{\s_1, \dots, \s_l}^{(l)} \approx \frac{L}{K} \be_c N - O(\sqrt{N}),
\eeq
for {\it all} $L=1 \dots K-1$. This is the entropic repulsion phenomenon which plays a fundamental role 
in the behavior of the system. It should be clear by now what the physical principles behind \eqref{entropic conditions} are. If the underlying strategy were to be maximizing on all levels (the so-called "greedy algorithm") the maximum would be lower (by a logarithmic factor) than the REM-maximum. 
In the critical case, however, the system has a better strategy: one looks for configurations 
$(\s_1, \dots, \s_{K-1})$ which are lower than the maximal possible {\it intermediate} values. This is energetically not optimal, but it allows us to meet an optimal balance between energy and entropy: there is in fact a large number of configurations satisfying \eqref{entropic conditions}, in fact: subexponentially many, and attached to these 
one still finds $\s_K's$  managing to make up for the energy loss.

\chapter{A multiscale refinement of the second moment method}
In this section we study the extremes of the Gaussian hierarchical field, which is the following model.

Let $N\in \N$. We consider binary strings $\alpha = \alpha_1 \alpha_2 \dots \alpha_N$,  $\alpha_i \in \{0,1\}$. $l(\alpha) = N$ is the length of the string. We write $T$ for the set of all such strings, and $T_N \subset T$ for the set of strings of length $N$.  If $\alpha \in T_N, 0\leq k \leq N$, we write $[\alpha]_{k}$ for the substring
$
[\alpha_1 \alpha_2 \dots \alpha_N]_{k} = \alpha_{1}\dots \alpha_{k}.
$
The Gaussian hierarchical field $(X_\alpha)_{\alpha \in T}$ is a family of Gaussian random variables where, 
for $\alpha \in T$ and $l(\alpha) = N \geq 1$,
\beq \label{hierarchical_field}
X_{\alpha} = \xi_{\alpha_1}^1+\xi_{\alpha_1, \alpha_2}^2+\dots+\xi_{\alpha_1 \alpha_2 \dots \alpha_N}^N\,,
\eeq
Here $(\xi_{\alpha_1 \alpha_2 \dots \alpha_i}^i; i\in \N, \alpha_1 \dots \alpha_i \in T_i)$ is a family of 
independent, centered normally distributed random variables with unit variance.  In other words, the Gaussian hierarchical field is a critical GREM$(K)$ with $K = K(N) = N$, i.e. with a {\it growing} number of levels. 

So far, the first step in the analysis of the extremes has been the computation of the leading order of the maximum. The approach used to solve REM and GREM is however hardly conceivable here. In the hierarchical field the number of levels in the underlying tree grows indefinitely, so when implementing the method from the previous section we would ideally end up with an optimization problem such as \eqref{finitely_cons} with {\it infinitely} many constraints: this is definitely not a promising route. To circumvent this obstacle, we will tackle the issue by means of a multiscale refinement of the second moment method.  The main ingredient behind the refinement is a coarse graining scheme which implements the idea that hierarchical fields can be used as an approximation tool. (This idea pervades the whole Parisi theory for mean field spin glasses, 
but no knowledge of this is assumed here).  At an abstract level one may say that 
we approximate the hierarchical field with yet another hierarchical field with less structure, a critical GREM with K-levels, but in concrete terms we will simply approximate the maximum of the 
hierarchical field (a number)  with the maximum of a critical GREM(K). It will become clear that it is 
not essential that the target model, namely the one we wish to approximate through a GREM, is exactly hierarchically organized\footnote{it goes without saying, if this is the case technicalities are naturally reduced by an order of magnitude.}; all is needed is a certain phenomen of {\it decoupling at mesoscopic scales}.

\section{The leading order of the maximum of the hierarchical field}

Again, we shorten $\be_c = \sqrt{2\log 2}$. By Markov inequality (a union bound) and standard Gaussian estimates one gets that to $\epsilon > 0$ there exists $\de = \de(\epsilon)>0$ such that
\beq \label{upper}
\PP\left[ \max_{\alpha\in T_N} X_{\alpha} \geq \be_c(1+\epsilon) N \right] \leq e^{- \de(\epsilon) N}
\eeq
for large  enough $N$. Just as in the case of a critical GREM, this simple bound turns out to be tight
(despite correlations):  
\begin{prop} \label{lower_prop} Given $\eta >0$ there exists $b = b(\eta)>0$ such that
\beq \label{lower_bound}
\PP\left[ \max_{\alpha\in T_N} X_{\alpha} \leq \be_c(1-\eta) N \right] \leq e^{-b N}.
\eeq
\end{prop}
Combining with \eqref{upper}, one therefore gets 
\beq \label{field_trivial}
\lim_{N\to \infty} \frac{1}{N} \max_{\alpha\in T_N} X_{\alpha} = \be_c,
\eeq
exactly as in the REM. The result \eqref{field_trivial} is well-known, with the first proof (through martingale techniques) seemingly due to Biggins \cite{biggins}. The method of proof I am presenting here seems however to be new. It rests on elementary considerations only.  

\begin{proof} {\bf (The multi-scale refinement: proof of Fact \ref{lower_prop})}
Pick a number $K\in \N$, which will play the role of the number of scales in the procedure. (The correct choice, i.e. the number of levels which do the job for approximating
the field,  will eventually depend on the level of precision $\eta$.) Assuming without loss of generality that $N/K$ is an integer, we split the strings of length $N$ into $K$ blocks of length $N/K$. More precisely, 
with $j_l \defi l N/K$ for $l\in \N$, and for $\alpha \in T_N$ we write 
\beq \label{multiscale}
X_{\alpha} = \sum_{l = 1}^K X^l_{\alpha}, \quad \quad X^l_{\alpha} \defi \sum_{i = j_{l-1}+1}^{j_{l}} \xi_{\alpha_1, \dots, \alpha_i}^i.
\eeq
This should be viewed as a "coarse graining" of the field. Remark that the $X^l_{\alpha}$ are centered Gaussians of variance $N/K$ but they are no longer, in general, independent. 
 Finally, we introduce the counting random variable
\beq \label{getgoing}
\mathcal N_N(\la_2, \dots, \la_K) \defi \sharp\left\{ \alpha \in T_N:  X^l_{\alpha} \geq \lambda_l N,\, l=2\dots K \right\}.
\eeq
The mystery behind the reason for {\it not} considering the first level $l=1$ will gradually disappear
in the course of the proof. (At the risk of being opaque: not considering the first levels allows us to gain independence, breaking the self-similarity of the field.) 

The claim is that for given $\vare >0$ and $\la_l \leq \la(\vare, K), \,l=2, \dots K$ where
\beq \label{cond}
\la(\vare, K) \defi \frac{\be_c}{K} (1-\vare) 
\eeq
there exists $\de = \de(\vare, K)$ such that
\beq \label{lower}
\PP\left[ \mathcal N_N(\la_2, \dots, \la_K) > 0 \right] \geq 1 - e^{-\de N}.
\eeq
This is the fundamental estimate. We postpone its proof for the moment and show how it steadily yields the desired result. Shortening $X_{\alpha}(2) \defi\sum_{l=2}^K X_{\alpha}^l$, by
\eqref{lower}, 
\beq  \label{max_mod}
\max_{\alpha \in T_N} X_{\alpha}(2) \geq \sum_{l=2}^K \la(\vare, K)N = \be_c \left(1-\frac{1}{K} \right)(1-\vare)N,
\eeq
on a set of probability greater than $1 - e^{-\de N}$. Now,
\beq \bea \label{main_task}
& \PP\left[ \max_{\alpha \in T_N} X_\alpha \leq \be_c (1-\eta) N\right] \leq \PP\left[ \max_{\alpha \in T_N} X_{\alpha}(2) < \be_c \left(1-\frac{1}{K} \right)(1-\vare)N \right] + \\
& \hspace{2cm} + \PP\left[ \max_{\alpha \in T_N} X_\alpha \leq \be_c (1-\eta) N, \; \max_{\alpha \in T_N} X_{\alpha}(2) \geq \be_c \left(1-\frac{1}{K} \right)(1-\vare)N \right]\,.
\eea \eeq
The first probability on the r.h.s is exponentially small by \eqref{max_mod}. On the other hand, the event for the 
second probability is included in the event that 
\beq \label{getting_rid}
\min_{[\alpha]_1 \in T_{N/K}} X_{\alpha}^1 \leq \be_c\left\{ (1-1/K)(1-\vare)-(1-\eta)\right\} N = - \be_c \Delta N,
\eeq
where
$
\Delta = \Delta(\eta, \vare, K)  \defi \eta-\vare (1-1/K).
$
By symmetry, 
\beq \bea \label{symmetry}
\PP\left[ \min_{[\alpha]_1 \in T_{N/K}} X_{\alpha}^1 < - \be_c \Delta N \right] & = \PP\left[ \max_{[\alpha]_1 \in T_{N/K}} X_{\alpha}^1 > \be_c \Delta N \right] \\
& \leq 2^{N/K} \PP\left[ X_{\boldsymbol{1}}^1 > \be_c \Delta N \right] \\
& \lesssim \exp N \log 2 \left( \frac{1}{K} - K \Delta^2 \right)
\eea \eeq
Choosing $K= 2/\eta, \,\vare = \eta/2$ we have $\frac{1}{K} - K \Delta^2 < 0$ hence the first probability in \eqref{main_task} is also exponentially small,
settling claim \eqref{lower_bound}. 

Claim \eqref{lower} follows from the {\it Paley-Zygmund inequality}. For this one needs to control first- and second moments. Clearly,  
\beq\bea \label{first_mom}
\mathbb{E}\left[ \mathcal N_N(\la_2, \dots, \la_K)\right] &= 2^N \PP\left[X^l_{{\boldsymbol 1}} \geq \lambda_l N,\, l=2\dots K \right] \\
& = 2^N \prod_{l=2}^K \PP\left[X^l_{{\boldsymbol 1}} \geq \lambda_l N \right] \\
& \cong 2^N \exp\left( - \frac{NK}{2} \sum_{l=2}^K \la_l^2 \right).
\eea \eeq
For two strings $\alpha, \alpha'$ and $ l=2, \dots, K$, shorten $p_l(\alpha, \alpha') \defi \PP\left[X_{\alpha}^l \geq \la_l N, X_{\alpha'}^l \geq \la_l N \right]$
(omitting the $\la$-dependence for ease of notation) and write
\beq \bea \label{sec_sum}
\mathbb{E}\left[\mathcal N_N^2(\la_2, \dots, \la_K)\right] &= \sum_{\alpha, \alpha' \in T_N\times T_N} \prod_{l=2}^K  p_l(\alpha, \alpha') \\
& = \sum_{[\alpha]_{j_1}\neq  [\alpha']_{j_1}} + 
\sum_{r=1}^{K-1} \sum_{\substack{[\alpha]_{j_r}=  [\alpha']_{j_r}\\
[\alpha]_{j_{r+1}}\neq [\alpha']_{j_{r+1}}}} 
\eea \eeq
The first sum on the r.h.s above is less than $\mathbb{E}\left[\mathcal N_N(\la_2, \dots, \la_K)\right]^2$; 
recovering this term is the crucial reason for not considering the first level in the definition of the counting random
variable \eqref{getgoing}. 

Let now $1 \leq r \leq K-1$, and consider two strings $\alpha, \alpha'$ satisfying
\beq \label{branching_off}
[\alpha]_{j_r}=  [\alpha']_{j_r},\qquad \text{but}\qquad [\alpha]_{j_{r+1}}\neq [\alpha']_{j_{r+1}}.
\eeq
The number of such couples of strings is {\it at most}
\beq \label{pathcounting}
2^{r N/K} \times 2^{2 N/K} \times 2^{2 (K-r-1) N/K} = 2^{(2K-r)N/K}.
\eeq
Moreover, for $\alpha, \alpha'$ such that $[\alpha]_{j_l}=  [\alpha']_{j_l}$ ($l\leq r$) one has
\beq \label{eq}
p_l(\alpha, \alpha') = \PP\left[ X_{\alpha}^l \geq \la_l N \right] \cong \exp\left( - \frac{KN}{2} \la_l^2\right),
\eeq
\beq \label{bound}
p_{r+1}(\alpha, \alpha') \leq \PP\left[ X_{\alpha}^{r+1} \geq \la_{r+1} N \right] \cong \exp\left( - \frac{KN}{2} \la_{r+1}^2\right)
\eeq
and, by independence, 
\beq \label{ind}
p_{l}(\alpha, \alpha') \cong \exp\left( - KN \la_{l}^2\right), \quad l= r+2, \dots, K.
\eeq
The last property \eqref{ind} is an instance of the aforementioned phenomenon of {\it decoupling at mesoscopic scales}. In virtue of the underlying 
correlation structure, the decoupling holds for the Gaussian hierarchical field exactly. 

The bound \eqref{bound} is the {\it REM-approximation}; for this the underlying  structure plays no role as one simply proceeds by ``worst case scenario''.

Using \eqref{eq}-\eqref{ind} in \eqref{sec_sum} we obtain
\beq \bea \label{second_mom}
& \mathbb{E}\left[\mathcal N_N^2(\la_2, \dots, \la_K)\right] \lesssim \mathbb{E}\left[\mathcal N_N(\la_2, \dots, \la_K)\right]^2+\\
& \qquad + \sum_{r=1}^{K-1} 2^{(2K-r)N/K} \exp\left(- \frac{KN}{2} \sum_{j=2}^{r+1} \la_j^2- KN \sum_{j=r+2}^K \la_j^2\right),
\eea \eeq
with the last sum set to zero when meaningless. Introducing 
\beq
\mathfrak{R}_m = \sum_{r=1}^{K-1} \exp N \left(\frac{K}{2} \sum_{j=2}^{r+1} \la_j^2 - \frac{r}{K} \log 2\right),
\eeq
it follows from Paley-Zygmund, \eqref{second_mom}, and the elementary estimate $\frac{1}{1+a} \geq 1-a$, that 
\beq \bea \label{PZ}
\PP\left[\mathcal N_N(\la_2, \dots, \la_K)>0 \right] & \geq \frac{\mathbb{E}\left[\mathcal N_N(\la_2, \dots, \la_K)\right]^2}
{\mathbb{E}\left[\mathcal N_N(\la_2, \dots, \la_K)^2\right]} \geq \frac{1}{1+\mathfrak{R}_N} \geq 1- \mathfrak{R}_N.
\eea \eeq
One readily checks that $\mathfrak{R}_N$ is exponentially small provided $\la_l < \frac{\be_c}{K}$ for $l=2, \dots, K$,
settling \eqref{lower}.  

\end{proof}

\subsection{Ingredients for a general recipe}
The method described above can be applied to a number of models. We shall thus dwell on its main steps. The discussion is intentionally informal.\\

{\sf Step 0.}  A general thread when addressing the leading order (but it is also a good idea 
whenever the finer properties are concerned) is to {\it always} give oneself an "epsilon of room", such as in  the $\vare$ in \eqref{upper}, or the $\eta$ in \eqref{lower_bound}.\\

{\sf Step 1.} If one is interested in the leading order, the first step for models in the REM-class is 
always easy: the upper bound to the leading order should always follow from a certain Markov inequality (union bound), cfr. \eqref{upper}. \\

{\sf Step 2.} As for a lower bound to the leading order, 
the multiscale refinement seems particularly efficient. This relies on a number of intermediate steps: 
\begin{itemize}
\item Coarse graining, such as \eqref{multiscale}. In the particular case of the hierarchical field, 
this is particularly simple since one can easily "visualize" the levels.
\item Breaking of the self-similarity, e.g. by not considering the first level as in \eqref{getgoing}. In the hierarchical field, based on the crucial decoupling which happens from a certain level downwards, this step is easy. In general, the idea behind this step is to "gain independence". 
\item {\bf REM approximation}. This is the arguably the deepest point. One simply {\it drops} correlations within the scales, cfr. \eqref{bound}. Needless to say, this leads to a 
dramatic simplification of the computations. The method is effective because we are not forgetting correlations globally, but only {\it within} a given scale. The error in the approximation is eventually due to the fact that one proceeds, within scales, by "worst case scenario". On the other hand, if a higher level of precision is sought, one simply increases the number of scales. 
\end{itemize}

{\sf Step 3.} Paley-Zygmund, \eqref{PZ}. This inequality is among the few truly universal methods available. It plays no specific role whether the underlying random variables are Gaussian (although this feature makes computations particularly straightforward).\\

It should be clear that minor modifications of the path outlined above allow to address quantities such as free energy, entropy, etc. I will not go into that.

\section{Beyond the leading order}

\subsection{Matching}
Once the leading order of the hierarchical field has been addressed, one may move to the subleading 
orders. For this, the method of matching allows us to make educated guesses. It is natural to expect that 
the level of the maximum of the hiearchal field is given by
\beq 
a_N \defi \be_c N + \omega_N,
\eeq
where $\omega_N = o(N)$ as $N\to \infty$. The question is of course how to  choose $\omega_N$. Consider to this end the extremal process of the hierarchical field
\beq
\Xi_N \defi \sum_{\alpha \in T_N} \de_{X_\alpha - a_N}
\eeq
A naive use of the method of matching, namely requiring that 
$\mathbb{E} [\Xi_N(A)]$ (for compact $A\subset \R$) remains of order one in the limit $N\to \infty$, yields
\beq \label{wrong guess}
\omega_N = -\frac{1}{2 \be_c}\log N,
\eeq
just as in the REM case. This turns out to be {\it wrong}. By linearity of the expectation, in $\mathbb{E}[\Xi_N(A)]$ we are completely dismissing
the underlying correlations, but these are, in the case of the hierarchical field, too severe. In fact, they are severe enough to have an impact already detectable at the level of the maximum (contrary to the critical GREMs, where they only reduce the intensity of the limiting extremal process). 

A key observation is that the approach for the leading order, see in particular \eqref{getgoing} and \eqref{cond}, suggests that extremal configurations must necessarily satisfy
\beq 
X_{\alpha}^{l} \leq \frac{\be_c}{K} N, \quad l=1\dots K-1,
\eeq
($K$ is the number of levels in the coarse graining), which would then imply that
\beq
\sum_{i=1}^l X_{\alpha}^{l} \leq \frac{l}{K} \be_c, \quad l=1 \dots K-1.
\eeq
In the (ideal) limit of $N$ levels of coarse-graining, for an $\alpha$-configuration to contribute to the extremal process it must hold that
\beq \label{guessguess}
\sum_{i=1}^l \xi_{\alpha_1, \dots, \alpha_i}^{i} \leq l \be_c .
\eeq
As in the critical GREM$(K)$, let us thus consider a thinned version of the 
extremal process,
\beq
\Xi_N^{(\leq)} \defi \sum_{\alpha}^{(\leq)} \de_{X_\alpha - a_N},
\eeq
where $\sum^{(\leq)}$ refers to those configurations satisfying \eqref{guessguess}. Following the method of matching, we shall require that for given compact $A \subset \R$, 
\beq
\mathbb{E}\left[ \Xi_N^{(\leq)}(A)\right] \sim 1 \quad (m\to \infty).
\eeq
Shortening $\boldsymbol 1 = (1, \dots, 1)$ for the unit string of length $N$, we have
\beq \bea \label{rematching}
&\mathbb{E}\left[ \Xi_N^{(\leq)}(A)\right]  = 2^N \PP\left[
\sum_{i=1}^l \xi_{ [\boldsymbol 1]_i}^{i} \leq l \be_c\; (l  \leq N-1),\;  X_{\boldsymbol 1} - a_N \in A\right]\\
& = 2^N \int_{A}\PP\left[
\sum_{i=1}^l \xi_{ [\boldsymbol 1]_i}^{i} \leq l \be_c\; (l  \leq N-1) \Big| X_{\boldsymbol 1} = a_N +x \in A\right] \exp\left( -\frac{(x+a_N)^2}{2N}\right)\frac{dx}{\sqrt{2\pi N}}. 
\eea \eeq
By the usual expansion, the term of the Gaussian density yields a contribution 
\beq \label{usual}
\sim 2^{-N} \exp\left(-\be_c x\right) \frac{\exp - \omega_N \be_c}{\sqrt{N}}.
\eeq
In order to get a handle on the conditional probability, we use a similar trick as in the critical  GREM$(2)$, i.e. we shift the sum by $\frac{l}{N} X_{\boldsymbol 1} $ so that the event becomes
\beq \bea
\sum_{i=1}^l \xi_{ [\boldsymbol 1]_i}^{i}   - \frac{l}{N} X_{\boldsymbol 1}     \leq l\be_c - \frac{l}{N} X_{\boldsymbol 1} \,.
\eea \eeq
Inspection of the covariance then shows that the Gaussian vector 
\beq 
\left( \sum_{i=1}^l \xi_{ [\boldsymbol 1]_i}^{i}   - \frac{l}{N} X_{\boldsymbol 1}\right)_{l=1, 2, \dots, N-1}
\eeq is independent of $X_{\boldsymbol 1}$. Using this, we get
\beq \bea
& \PP\left[\sum_{i=1}^l \xi_{ [\boldsymbol 1]_i}^{i} \leq l \be_c,  \; l  \leq N-1 \Big| X_{\boldsymbol 1} = a_N +x \right] \\
& = \PP\left[\sum_{i=1}^l \xi_{ [\boldsymbol 1]_i}^{i}   - \frac{l}{N} X_{\boldsymbol 1}     \leq l \be_c - \frac{l}{N} X_{\boldsymbol 1}, l \leq N-1 \Big| X_{\boldsymbol 1} = a_N +x\right] \\
& = \PP\left[\sum_{i=1}^l \xi_{ [\boldsymbol 1]_i}^{i}   - \frac{l}{N} X_{\boldsymbol 1}     \leq l \be_c - \frac{l}{N} (x+ a_N), l \leq N-1 \Big| X_{\boldsymbol 1} = a_N +x \right] \\
& = \PP\left[\sum_{i=1}^l \xi_{ [\boldsymbol 1]_i}^{i}   - \frac{l}{N} X_{\boldsymbol 1}     \leq l \be_c - \frac{l}{N} (x+ a_N), l \leq N-1 \right],
\eea \eeq
the last step by independence. But by definition $a_N = \be_c N +\omega_N$, hence the above 
simplifies to 
\beq \label{bridge!}
\PP\left[\sum_{i=1}^l \xi_{ [\boldsymbol 1]_i}^{i}   - \frac{l}{N} X_{\boldsymbol 1}     \leq - \frac{l}{N} (x+ \omega_N), l \leq N-1 \right]
\eeq
As it turns out, the law of the process $l\mapsto \sum_{i=1}^l \xi_{ [\boldsymbol 1]_i}^{i}   - \frac{l}{N} X_{\boldsymbol 1}, l =1, 2, \dots, N$ coincides with that of a (discrete) Brownian bridge 
of lifespan $N$ observed at the times $l=1, 2, \dots, N$. In particular, \eqref{bridge!} is 
the probability that a (discrete) Brownian bridge stays below the line $l\mapsto -(l/N)(x+\omega_N)$ 
at the observation-times $l=1, 2, \dots, N-1$. Such probability can be computed in different ways (e.g. using the reflection principle): the upshot is 
\beq
\PP\left[\sum_{i=1}^l \xi_{ [\boldsymbol 1]_i}^{i}   - \frac{l}{N} X_{\boldsymbol 1}     \leq - \frac{l}{N} (x+ \omega_N), l \leq N-1 \right] \sim \frac{1}{N},
\eeq
which is  {\it vanishing} in the limit $N\to \infty$. This has a dramatic consequence on the matching. 
In fact, combining this asymptotics with \eqref{usual} in \eqref{rematching} we see that
\beq
\mathbb{E}\left[ \Xi_N^{(\leq)}(A)\right] \sim \frac{1}{N}\frac{\exp - \omega_N \be_c}{\sqrt{ N}} \int_A e^{-\be_c x} dx \quad (N\to \infty),
\eeq
and this remains of order one in the considered limit only for
\beq
\omega_N = - \frac{3}{2\be_c} \log N.
\eeq
Summarizing, the above suggests that the level of the maximum in the hierarchical field is 
\beq 
a_N= \be_c N - \frac{3}{2\be_c} \log N.
\eeq

This guess turns out to be correct. I will definitely not give a proof of this here, as it requires some heavy technicalities,  but refer the reader to e.g. \cite{bramson} for details. On the other hand, 
I will sketch below the main steps behind a general approach to the subleading order for models in the REM-class. 

\begin{rem} A complementary route to these issues is provided by Fyodorov, Le Doussal and Rosso  \cite{fyodorov_rosso_doussal} within the multi-fractal formalism.
The idea consists in addressing the {\em density} of the counting random variable $\mathcal N_N(\la)$. One can then show (through the analysis of the moments) that, for specific values of $\la$, this density displays a power-law tail with $\la$-dependent exponent: by "matching to unity" (very much in the spirit of the method discussed above), one can then derive important information, such as the level of the maximum of the field.
\end{rem}

\subsection{How to get started: multi-scale Markov} \label{smart markov}
By matching, we have thus made a natural guess for the level of the maximum. However, when trying to make this rigorous, one faces an immediate difficulty. In fact, a most natural step in a rigorous treatment would be to show that, with $a_N = \be_c N- (3/2 \be_c) \log N$ as above, 
\beq \label{first diff}
\lim_{Y\to +\infty} \sup_{N} \PP\left[\exists \alpha \in T_N: X_\alpha - a_N \geq Y\right] = 0. 
\eeq
This would at least suggest that we are indeed on the right scale. A naive attempt to check \eqref{first diff} by union bounds and Markov inequality yields
\beq \bea
 \PP\left[\exists \alpha \in T_N: X_\alpha - a_N \geq Y\right] & \leq \frac{2^N}{Y} \int_{Y}^\infty 
\exp\left( -\frac{(x-a_N)^2}{2N}\right)\frac{dx}{\sqrt{2\pi N}} \\
& \sim \frac{N}{Y} \int_{Y}^\infty e^{-\be_c x} dx,
\eea \eeq
(the last asymptotics by the usual expansion of the Gaussian density), and this explodes as 
$N\to \infty$. In other words, a plain application of Markov inequality is inconclusive.
This is of course due to the logarithmic correction which is larger than in the REM-case. 
{\it But if not Markov, what else?} A technically convenient way out is to first give 
ourselves an epsilon of room, i.e. to consider 
\beq 
a_N^{(\vare)} \defi \be_c N - \frac{3-\vare}{2 \be_c} \log N,
\eeq
for some $\vare>0$. Remark that this is higher than the alleged level of the maximum. 
The goal is to prove that with overwhelming probability, no configuration will be found reaching these heights, and this will imply that $a_N = a_N^{(\vare=0)}$ is at least an upper bound to the level of the maximum of the hierarchical field. For this, we will rely on a 
Markov-type inequality which keeps track of the mutliple scales. 

More precisely, consider the discrete function 
\beq 
l \mapsto f_N(l) =  \be_c l + K \log N
\eeq
for some large $K$ which will be identified later. Remark that this is just a "perturbation" of the linear function $\frac{(\cdot)}{N} \be_c$ which arises in the analysis of the leading order of the maximum. The form of the perturbation is not really important, and $K \log N$ is only a convenient choice. 
The claim is now that for $\delta>0$ and large enough $K$,
\beq \label{smartmark one}
\PP\left[ \exists l\leq N, \alpha \in T_l: \; X_{[\alpha]_l} \geq f_N(l)\; \text{for some}\; l \in [1, N-1] \right] \leq \delta. 
\eeq
The proof of this fact is elementary, and goes by Markov's inequality. It holds: 
\beq \bea
& \PP\left[ \exists \alpha \in T_N: \; X_{[\alpha]_l} \geq f_N(l)\; \text{for some}\; l \right] \\
& \leq \sum_{l=1}^{N-1} 2^l \PP\left[ X_{[\boldsymbol 1]_l} \geq f_N(l) \right] \\
& \leq \sum_{l=1}^{N-1} 2^l \int_{0}^\infty \exp\left(-\frac{(x+f_N(l))^2}{2l} \right) \frac{dx}{\sqrt{2\pi l}}.
\eea \eeq
By the usual quadratic expansion, and elementary bounds, the above is easily seen to be
\beq 
\sim \sum_{l=1}^{N-1}  \frac{1}{\sqrt{l}} \exp\left( - \be_c K \log N\right), 
\eeq
which can be made as small as wished by choosing $K$ large enough. This proves \eqref{smartmark one}. 

Equation \eqref{smartmark one} is an important piece of {\it a priori} information. In fact, it will allow us to prove that for given $\vare>0$ the probability 
that there exists an $\alpha \in T_N$ such that  $X_{\alpha} \geq a_N^{(\vare)} $ is vanishingly small. 
To see this, let us introduce, for $\alpha \in T_N$, the event 
\beq 
B^{f_N}[\alpha] = \left\{ X_{[\alpha]_l} \leq f_N(l), \; 0 \leq l \leq N   \right\}.
\eeq
We then write 
\beq \bea
& \PP\left[\exists \alpha \in T_N: X_{\alpha} \geq a_N^{(\vare)} \right] = 
\PP\left[\exists \alpha \in T_N: X_{\alpha} \geq a_N^{(\vare)},  B^{f_N}[\alpha] \right] +\\
& \hspace{3cm} + \PP\left[\exists \alpha \in T_N: X_{\alpha} \geq a_N^{(\vare)},  B^{f_N}[\alpha]^\texttt{c} \right]
\eea \eeq
By \eqref{smartmark one}, the second probability on the r.h.s. can be made as small as wished, 
so it remains to address the first term. Proceeding by union bound and conditioning, we get 
\beq \bea \label{proceeding by}
& \PP\left[\exists \alpha \in T_N: X_{\alpha} \geq a_N^{(\vare)},  B^{f_N}[\alpha] \right] \\
& \quad \leq 2^N \int_{} \PP\left[B^{f_N}[\boldsymbol 1] \Big|  X_{\boldsymbol 1} = x+ a_N^{(\vare)} \right] 
\exp\left( - \frac{(x+a_N^{(\vare)})^2}{2 N}\right) \frac{dx}{\sqrt{2\pi l}}
\eea \eeq
The usual Gaussian estimates, recalling that $a_N^{(\vare)} = \be_c N - (3-\vare)/2 \be_c \log N$, 
yield \beq \label{always the same}
2^N \exp\left( - \frac{(x+a_N^{(\vare)})^2}{2 m}\right) \frac{dx}{\sqrt{2\pi l}} \sim 
N^{1-\vare}  e^{-\be_c x} ,
\eeq
in first approximation. So it remains to get a handle on the conditional probability:
\beq \bea \label{intermediate}
& \PP\left[B^{f_N}[\boldsymbol 1] \Big|  X_{\boldsymbol 1} = x+ a_N^{(\vare)} \right]  \\
& \quad = \PP\left[ X_{[\boldsymbol 1]_l} \leq f_N(l), l=1\dots N-1  \Big|  X_{\boldsymbol 1} = x+ a_N^{(\vare)} \right] 
\eea \eeq
It is easy to see that the process $l \mapsto X_{[\boldsymbol 1]_l}$ {\it conditioned} on the terminal point is a (discrete) Brownian bridge with drift. More precisely, 
with $\mathfrak Z_N(l)$ a Brownian bridge of lifespan $N$, starting and ending in zero, and $\mathcal L$ denoting the law of such process, 
\beq \label{law}
\mathcal L\left( X_{[\boldsymbol 1]_l}, l\leq N \;\Big| \;
X_{\boldsymbol 1}= x+ a_N^{(\vare)} \right) = \mathcal L\left( \mathfrak Z_N(l) + \frac{l}{N}(x+a_N^{(\vare)}), l\leq N\right)
\eeq
hence 
\beq \bea
\eqref{intermediate} & = \PP\left[ \mathfrak Z_N(l) + \frac{l}{N}(x+a_N^{(\vare)}) \leq 
f_N(l), l = 1,\dots,  N-1 \right] \\
& \leq \PP\left[ \mathfrak Z_N(l) \leq 
K' \log N, l = 1,\dots,  N-1 \right]
\eea \eeq
(since $\geq 0$) and where $K' \defi K + \frac{(3-\vare)}{2 \be_c} $.
One can get a handle on the above probability with the reflection principle: the upshot is 
\beq \label{upshot bbridge}
\PP\left[ \mathfrak Z_N(l) \leq K' \log N, l = 1\dots N-1 \right] \sim \frac{(\log N)^\delta}{N},
\eeq
for some $\de = \de(K')$. The crucial point is that this behaves 
as $1/N$ up to {\it logarithmic} (in $N$) corrections. Using this and \eqref{always the same}  we get that  \eqref{proceeding by} 
is vanishingly small thanks to the additional term $N^{-\vare}$. (This is the ultimate reason 
for giving oneself "an epsilon of room"). Since this holds for any $\vare$, we have thus proved that 
\beq \label{upper bound full}
\max_{\alpha \in T_N} X_\alpha \leq \be_c N - \frac{3}{2\be_c}(1+o(1)) \log N,
\eeq
with overwhelming probability. 

\begin{rem}
Gaussianity, though naturally useful for computations, is not 
crucial for the method to work. In order to compensate
\eqref{always the same} one needs that the conditioned process \eqref{intermediate} "behaves" as a Brownian bridge, in the sense that 
the probability of staying below straight lines (or, more generally: envelopes) behaves to leading order as $1/N$. This is a delicate technical point: the level of precision required cannot be achieved by straightforward applications of, say, the "Hungarian theorems" \cite{hungarian}. To my knowledge, the best (and most flexible) result available in this direction is \cite{ford}.
\end{rem}

\subsection{Entropic repulsion and the restricted Paley-Zygmund}
Once the upper bound \eqref{upper bound full} has been established, it remains to identify a 
matching lower bound (up to $o(\log N)$-terms, say).  The main idea here is the physical principle of entropic repulsion, combined with a restricted Paley-Zygmund's inequality. In fact, it emerges from the above analysis that the "path" of an extremal configuration lies typically below the line $l \mapsto f_N(l)$. As we have seen, conditioning on the terminal point gets rid of the drift, and one ends up with a (discrete) Brownian bridge staying below a straight line for most of the time. It is well known that the strategy used by such Brownian bridge
to avoid hitting the line is to "go negative". To formulate this, we shall use a terminology introduced in \cite{abk},  that of {\it entropic envelope}. To define this, pick $0 < \gamma < 1/2$ and let 
\beq
E_{N, \gamma}(l) = \begin{cases}
 - l^\gamma, & 1 \leq l \leq N/2 \\
 -(N-l)^\gamma & N/2 \leq l \leq N. 
\end{cases}
\eeq
One can then prove (see \cite{abk}) that  
\beq 
\lim_{r\to \infty} \lim_{N\to \infty} \frac{\PP\left[ \mathfrak Z_N(l) \leq  K' \log N, l = r, \dots, N-r \right] }{\PP\left[ \mathfrak Z_N(l) \leq  E_{N, \gamma}(l), l = r, \dots, N-r \right] } = 1. 
\eeq
In other words, the probability that a Brownian bridge stays below a straight line is asymptotically the same as the probability of staying below the entropic envelope (as long as $\gamma< 1/2$, strictly). This suggests the following strategy to prove that $a_N$ is also a lower bound to the level of the maximum. Not surprisingly, we will first give ourselves an epsilon of room: let $\vare>0$ and
\beq 
a_N^{(-\vare)} \defi \be_c N - \frac{3+\vare}{2\be_c} \log N, 
\eeq 
Remark that this is {\it lower} than $a_N$. The goal is to prove that 
\beq 
\PP\left[ \max_{\alpha \in T_N} X_{\alpha} \geq a_N^{(-\vare)} \right] \to 1, \quad N\to \infty.
\eeq
For $r\in \{1, \dots N/2\}$, and $E$ the entropic envelope (to given $\gamma$) as above we write 
\beq
\max_{\alpha \in T_N \wedge E[r, N-r]} X_{\alpha}
\eeq
for the maximum {\it restricted} to those 
$\alpha \in T_N$ such that $X_{[\alpha]_l} \leq E_{N, \gamma}(l)$ for $l = r\dots N-r$. 
Furthermore, let 
\beq 
\mathcal Z_N^{(-\vare)} \defi \sum_{\alpha \in T_N \wedge E[r, N-r]} 1\left\{X_\alpha \geq a_N^{(-\vare)}\right\}
\eeq
We then have:
\beq \bea \label{pz}
\PP\left[ \max_{\alpha \in T_N} X_{\alpha} \geq a_N^{(-\vare)} \right] & \geq \PP\left[ \max_{\alpha \in T_N \wedge E[r, N-r]} X_{\alpha} \geq a_N^{(-\vare)} \right] \\
&\geq \frac{\mathbb{E}\left[ \mathcal Z_N^{(-\vare)} \right]^2 }{\mathbb{E}\left[ \left(\mathcal Z_N^{(-\vare)}\right)^2 \right]},
\eea \eeq
the first inequality since we consider the maximum over a smaller subset of $T_N$, and the 
second inequality by Payley-Zygmund. One can then show (essentially) that the r.h.s. of \eqref{pz} converges to one, in the limit $N\to \infty$ first (provided that $r\to \infty$ slowly). In other words, 
the mean of the second moment behaves asymptotically as the square of the first moment. The reason behind this is not difficult to understand, the main observation being already present in the 
critical GREM with two levels, see \eqref{cons_start}-\eqref{cons_finish}. In fact the second moment 
of $\mathcal Z_N^{(-\vare)}$ can be rewritten as a sum over all possible couples of configurations, 
and this sum can be re-arranged as a sum over the level up to which the two configurations 
coincide. Now: 
\begin{itemize}
\item either did the most recent common ancestor branch "very early", in which case the 
two configurations are (essentially) independent: this gives a contribution which is asympotically equivalent to the first moment squared; 
\item or the most recent common ancestor branched "late": in which case, having two configurations which reach the level of the maximum at level $N$ amounts 
to finding {\it within the same tree} attached to the most recent common ancestor {\it two} Brownian 
paths making the unusually large jump which brings them from the lower envelope to the level of the maximum. As shown in \eqref{smallsmall}, this probability is however small, hence the contribution to the second moment coming from couples with "late" common ancestor branching does not contribute to the second moment (cfr. \eqref{cons_finish}).
\end{itemize}

At a conceptual level the situation is therefore pretty simple: the picture is the same as that of the critical GREM(2). In particular, the main ideas are already laid out in the proof of  \eqref{claim_lower}.  Due to the "continuous" branching, the technical details in the case of the hierarchical field 
are however more demanding, so I will refer the reader to \cite[Lemma 11]{bramson}. 

\subsection{Interpolating between REM and hierarchical field, or: from 1 to 3}
Here is a model which, to my knowledge, has not been addressed in the mathematical literature. I am indebted to Bernard Derrida for discussions on this and related issues.  Loosely, the model interpolates between $K$ (GREM) and $N$ (hierarchical field). More precisely, pick $0 < \alpha < 1$ and consider a GREM with $N^\alpha$ 
levels: the configurations are $N^{\alpha}$-dimensional vectors, 
\beq
\alpha = (\alpha_1, \alpha_2, \dots, \alpha_{N^\alpha}), \qquad \forall_i: \alpha_i = 1\dots, 2^{N^{1-\alpha}} 
\eeq
and the energies are given by 
\beq
X_\alpha = X_{\alpha_1}^{(1)}+  X_{\alpha_1, \alpha_2}^{(2)}+ \dots + X_{\alpha_1, \dots, \alpha_{N^{\alpha}}}^{(N^\alpha)}
\eeq
where the $X_{\cdot}^{(i)}$ are all independent centered Gaussians with variance 
$N^{1-\alpha}$. 

By matching, and Brownian scaling, it is natural to conjecture 
that 
\beq \label{inter}
\max_{\alpha} X_\alpha = \be_c N - \frac{2\alpha+1}{2 \be_c} \log N + o(\log N), 
\eeq
with overwhelming probability. This problem is currently investigated in the PhD thesis 
of Marius Schmidt \cite{schmidt}. \\

Finite size corrections for the free energy of a GREM when the number of levels becomes infinite have been addressed by Cook and Derrida \cite{cook_derrida} by the use of traveling wave equations. This allows in particular to recover the logarithmic correction for the hierarchical field ($\alpha=1$
in the above setting). It would be interesting to have a proof of \eqref{inter} in full generality, and to understand the lower order corrections $o(\log N)$. This would open a path towards 
the extremal processes which arise in the large $N$ limit. One naturally expects the ensuing processes
to "interpolate" between the Poisson point process associated to the REM ($\alpha =0$, in which case the logarithmic correction equals $1$), and 
the Poisson cluster process first appeared in the context of branching Brownian motion/hierarchical field ($\alpha=1$, in which case the logarithmic correction equals $3$), see e.g. \cite{anton, gouere}. \\

Finally, it should be remarked that among all GREM-type models, the case $\alpha=1$ is the most correlated which still falls in the REM-class: one gets  the largest possible (logarithmic) correction while leaving the leading order "untouched". More severe correlations already affect the leading order, which is naturally expected under the light of, say, Slepian's lemma. At the level of the leading order, the situation has  been completely solved by Bovier and Kurkova \cite{anton_crem}. 
On the finer level, progress has been made  recently \cite{bovier_hartung, bovier_hartung_2, fang_zeitouni, maillard_zeitouni, mallein}, but much remains to be understood. There is a whole world of extremal {\it hierarchical} processes which still awaits to be discovered: according to the Parisi theory, these should be, to a vast extent, {\it universal}.

\chapter{Applications}

\section{The 2-dim Gaussian free field in a box}
Consider the $N\times N$ box $V_N \defi ([0, N]\cap \Z)^2$.  
$\partial V_N$ stands for its boundary on the two dimensional lattice. Denote by  
$w_m$ simple random walk started in $V_N$ and killed at time
$\tau = \min \left\{m: w_m \in \partial V_N\right\}$, when hitting the boundary. The Gaussian free field, GFF for short, is the 
mean zero Gaussian field $\{X_z^N \}_{z\in V_N}$ whose covariance is given by the Green function of the 
killed walk, 
\[
\mathbb{E}\left[X_x^N X_y^N\right] =  E^x\left[ \sum_{m=0}^\tau \1\{w_m = y\}\right]
\]
Shorten $g \defi 2/\pi$. By Markov inequality, for $\vare > 0$,  
\beq
\lim_{N\to \infty} \PP\left[ \max_{x\in V_N} X_x^N \geq  2 \sqrt{g} (1+\vare) \log N \right] = 0.
\eeq
Since this holds for arbitrary $\vare$, we have a first upper-bound
\beq \label{upper_gff}
\lim_{N\to \infty} \max_{x\in V_N} \frac{X_x^N}{\log N} \leq 2 \sqrt{g}.
\eeq
This simple {\it REM-bound}  turns out to be tight. 
 
\begin{prop}[Bolthausen, Deuschel, and Giacomin \cite{bdg}]
\[
\lim_{N\to \infty} \max_{x\in V_N} \frac{X_x^N}{\log N} = 2 \sqrt{g},
\]
in probability. 
\end{prop}
The $2^{nd}$-moment method described above provides a straightforward proof of this fact.  
(It also streamlines the more refined analysis by Daviaud \cite{daviaud} 
on the fractal structure of the sites where the GFF is large, but I will not go into that.)  The 
crucial observation \cite{bdg} is that the GFF admits a natural {\it multiscale decomposition}. 
To see how it goes, assume without loss of generality that $N=2^n$ and identify an integer $m = \sum_{l=0}^{n-1} m_i 2^i$ with its binary expansion $(m_{n-1}, m_{n-2}, \dots, m_0)$.
For $k\geq 1$ introduce the sets of $k-$diadic integers
$A_k = \{m\in \{1, \dots, N\}:\; m = (2l-1) N 2^{-k} \; \text{for some integer}\; l   \},$ and define the $\sigma$-algebras
$\mathcal A_k = \sigma\left(X_z^N: \; z = (x,y), \, x \, \text{or} \, y \in \cup_{i\leq k} A_i\right).$ For every $z= (x,y) \in V_N^o$ (the interior of $V_N$), write $z_i = (x_i, y_i)$ with 
$x_i, y_i$ denoting the $i$-th digit in the binary expansion of $x,y$. Introducing the random
variables $\xi_{z_{k+1}, \dots, z_n}^{z_1, \dots, z_k} \defi \mathbb{E}\left[X_z^N \mid \mathcal A_k \right]$, it holds
\beq  \label{first tel}
X_z^N = \xi^{z_{1}}_{z_2, \dots, z_k} +  \left(X_z^N - \xi^{z_{1}}_{z_2, \dots, z_k} \right) \defi \xi^{z_{1}}_{z_2, \dots, z_k} +     X_{z_2, \dots, z_n}^{z_1}.
\eeq
Furthermore, by the random walk representation of the covariance of the field one
sees that the collections $\left\{X_\cdot^{z_1}\right\}_{z_1 \in V_1}$ are i.i.d. copies of the GFF in the box $V_{N/2}$. I will refer to this feature as the self-similarity of the field. 
Thanks to the self-similarity, one may therefore iterate the 
procedure to get the equivalent of \eqref{multiscale}:
\beq \label{multiscale_full}
X_z^N = \xi^{z_{1}}_{z_2, \dots, z_k}+\xi^{z_{1}, z_2}_{z_3, \dots, z_k}+ \dots + \xi^{z_{1}, z_2, \dots, z_{N-1}}_{z_N} \,,
\eeq
with the summands on the r.h.s. being independent. Moreover, the field decouples at mesoscopic scales: 
$\xi^{z_{1}, \dots, z_k}_{z_{k+1}, \dots, z_n} \; \text{and}\; \xi^{w_{1}, \dots, w_k}_{w_{k+1}, \dots, w_n}$
are {\it independent} as soon as  $(z_{1}, \dots, z_k) \neq (w_{1}, \dots, w_k)$. This is   
the equivalent of \eqref{ind}, and the triviality of the free field.

The analogy with the hierarchical field runs even deeper. In fact, the full set of hierarchies is {\it not} required for 
the leading order of the maximum, with a coarse-grained version fullfilling the needs: choose $K$ such that $N/K$ is an integer, 
set $j_l = l N/K,$ where $l=1, 2,\dots$ and 
rewrite \eqref{multiscale_full} according to the coarser family of sigma-algebras $\mathcal A_{N/K}, 
\mathcal A_{2N/K}, \dots$ and write, 
\beq \label{klevels}
X_z^N = \sum_{l=1}^K \hat \xi^{z_1, \dots, z_{j_l}}_{z_{j_l+1}, \dots, z_N}
\eeq 
with $\hat \xi$ being defined in full analogy as above. Now, the $\hat \xi$-fields have a complicated correlation structure 
whenever two sites fall into the same box, but for this the REM-approximation provides an easy way out, and a rerun of \eqref{max_mod}-\eqref{PZ} 
immediately (up to a technicality which I am discussing below) 
settles the proof. 

The technical issue stems from certain boundary effects which are however mild for the 
purpose of establishing the leading order of the maximum. Indeed, 
the method presented in these notes
requires one quantitative input only: a fair control of the second moment of the underlying random variables.  
It is known \cite{bdg} that
\beq \label{upper_var}
\sup_{x\in V_N} \mathbb{E}\left[(X_x^N)^2 \right] \leq g \log N+c\,,
\eeq
but an equivalent lower bound breaks down when $x$ approaches the boundary $\partial V_{N}$. To surmount this
obstacle one picks $\de>0$ and introduces $V_N^\de \defi \left\{x \in V_{N}: \; \text{dist}(x, \partial V_{N}) \geq \de N \right\}$. 
It is known that for $\de \in (0, 1/2)$ there exists $c(\de)>0$
\beq \label{var_control}
g\log N - c(\de)\leq  \mathbb{E}\left[(X_x^N)^2 \right]\leq  g\log N + c(\de)
\eeq
for all $x\in V_N^\de$, see e.g. \cite{bdg}. Of course, such boundary effects arise at each step in a coarse-graining with, say, $K$ levels. 
For this it however suffices to ''stay away'' from the respective boundaries, and this can be achieved by simply
applying the method outlined above to the maximum restricted to a subset  $V_N^{\delta, K}\subset V_N$ only, such as
the shaded region in Figure \ref{figu} below. 
\begin{figure}[h] 
\begin{center}
\includegraphics[scale=0.4]{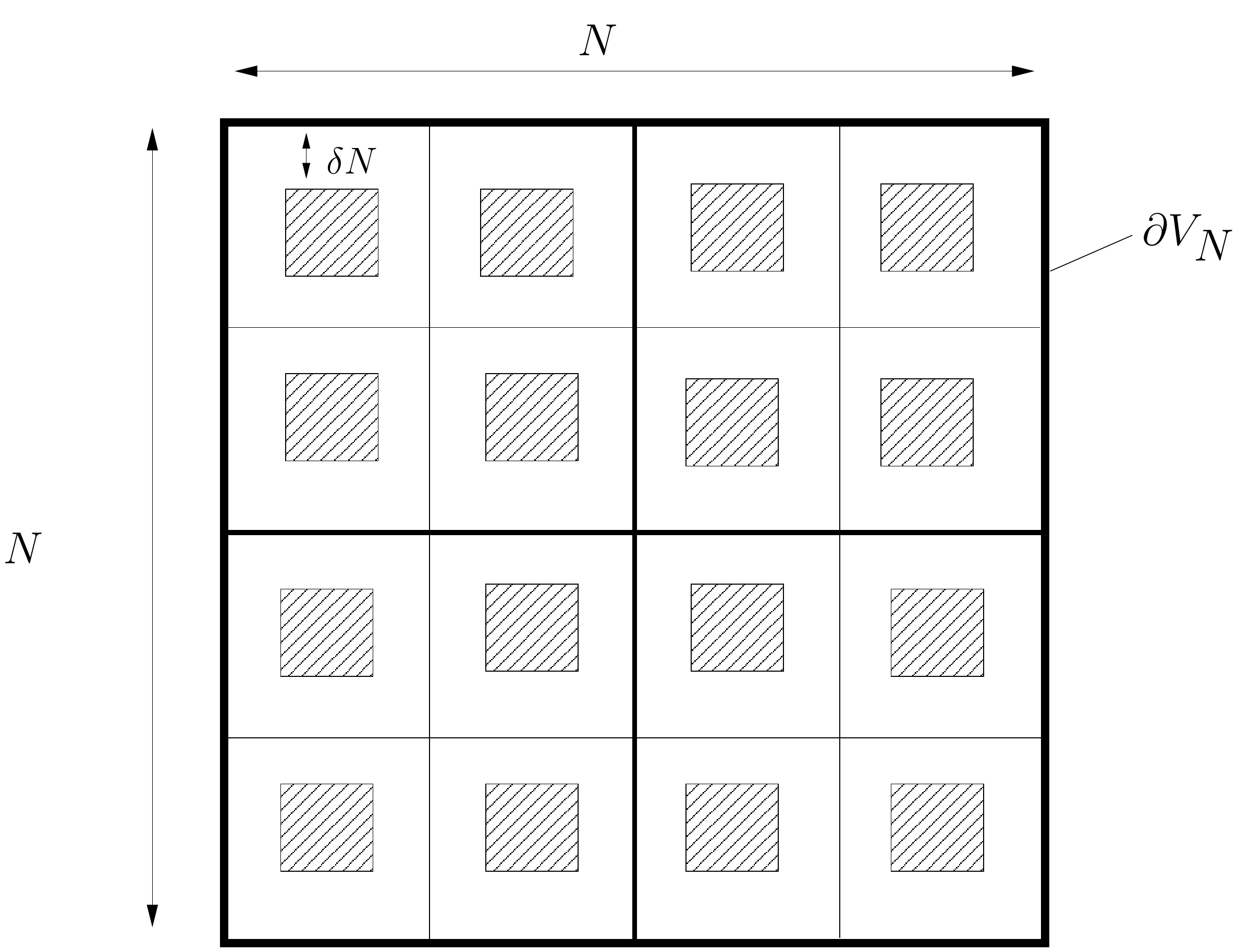}
\caption{Subset of sites for coarse graining with $K=3$ levels.}
\label{figu}
\end{center}
\end{figure}

The matching allows  us to make educated guesses on the subleading order of the maximum, and how the field manages to overcome correlations. To this end, let
\beq 
a_N \defi 2 \sqrt{g} \log N + C \log \log N. 
\eeq
Consider also the extremal process 
\beq 
\Xi_N \defi \sum_{x \in V_N} \de_{X_x^N - a_N}
\eeq
The naive approach would amount to requiring that for compact $A\subset \R$, 
\beq
\mathbb{E} \Xi_N(A) \sim 1, \quad N\to \infty,
\eeq
and simple computations show that this is equivalent to 
\beq
N^2 \exp\left( - \frac{a_N^2}{2 g \log N} \right) \frac{1}{\sqrt{\log N}} \sim 1, \quad N\to \infty.
\eeq
This would yield 
\beq 
a_N = 2 \sqrt{g} \log N - \frac{1}{4} \sqrt{g} \log \log N.
\eeq
This is of course wrong: we are completely dismissing correlations. Based on the analogy with the 
hierarchical field, ensuing from the multiscale decomposition \eqref{multiscale_full}, we shall require that "extremal sites" also satisfy the requirement 
\beq
\sum_{i=1}^l \xi^{z_{1}, \dots, z_{1, \dots i}}_{z_{i+1}, \dots, z_k} \leq 2 l \sqrt{g}  , \quad l= 1\dots, \log (N-1)\,,
\eeq
One can check that for a given site $z\in V_N$, the probability of this event under the conditioning that 
$X_z^{N} \approx a_N$, behaves to leading order as $1/\log N$ and therefore the matching would 
read 
\beq 
N^2 \frac{1}{\log N}\exp\left( - \frac{a_N^2}{2 g \log N} \right) \frac{1}{\sqrt{\log N}} \sim 1\,,
\eeq
which holds for 
\beq 
a_N = 2 \sqrt{g} \log N - \frac{3}{4} \sqrt{g} \log \log N \,.
\eeq
(Which is the same answer one gets for the maximum of BBM after suitable parametrization.)
This turns out to be correct; the reader is referred to the lecture notes \cite{zeitouni} and references therein for details.

\section{First/last passage percolation on the binary tree}
Gaussianity of the involved random variables is not essential for the method to
work. The following example illustrates that 
fairly good large deviations estimates only are needed. 
We consider again the binary tree $T_N$ and
\beq 
E_{\alpha} \defi e_{\alpha_1}+\dots + e_{\alpha_1, \dots, \alpha_N}
\eeq
where now the $e's$ exponentially distributed with mean one (and independent of
each other). We introduce Aldous' terminology of the {\it natural outer bound} \cite{aldous}, namely the 
$C_\star$ such that 
\beq \bea 
& 2^N \PP\left[ E_{\boldsymbol 1} > c N\right] \to 0 \qquad \text{for}\; c > C_\star\, ,\\ 
& 2^m \PP\left[ E_{\boldsymbol 1} > c N\right] \to \infty \qquad \text{for}\; c < C_\star.
\eea \eeq

Denoting by $I(x) \defi x-1-\log x$ the {\it rate function} of the exponential, 
we identify the outer bound as the solution to  
\beq
\log 2 = I(c), \quad \text{i.e.} \quad 2 c = e^{c-1}. 
\eeq
There are two solutions to this equation, $C_\star^1 < 1 < C_\star^2 < \infty$. The solution $C_\star^2$ corresponds to the last passage percolation (LPP) problem, whereas $C_\star^1$ is a first passage percolation (FPP). 
A rerun of the computations presented above (with obvious modifications for the FPP) 
immediately yields the following result:

\begin{prop} With the above notations,
\[
\lim_{N\to \infty} \frac{1}{N} \min_\alpha E_{\alpha} = C^1_{\star},\quad \lim_{N\to \infty} \frac{1}{N} \max_\alpha E_{\alpha} = C^2_{\star} 
\]
in probability.
\end{prop}
I am unable to track down the first place where this result appeared - but it is definitely already 
present in \cite{aldous}. The identification of the subleading orders is an application of the method of matching and is left to the reader as an instructive exercise.

\section{Percolation on the hypercube}
Here is an example where neither Gaussianity, nor hierarchical structure (not even at "mesoscopic" 
level, such is the case of the GFF) are available but which still falls in the REM-class. Consider the 
unit cube $\{0,1 \}^N$ in $N$ dimensions. To each edge we attach independent exponential random 
variables $\xi$ with mean one. 
We write $\boldsymbol 0 = (0,0, \dots, 0)$ and $\boldsymbol 1 = (1,1,\dots,1)$
for diametrically opposite vertices. As an example, the $(N=10)$-dimensional hypercube is shown in 
Figure \ref{diamond} below.

\begin{figure}[h] 
\begin{center}
\includegraphics[scale=0.4]{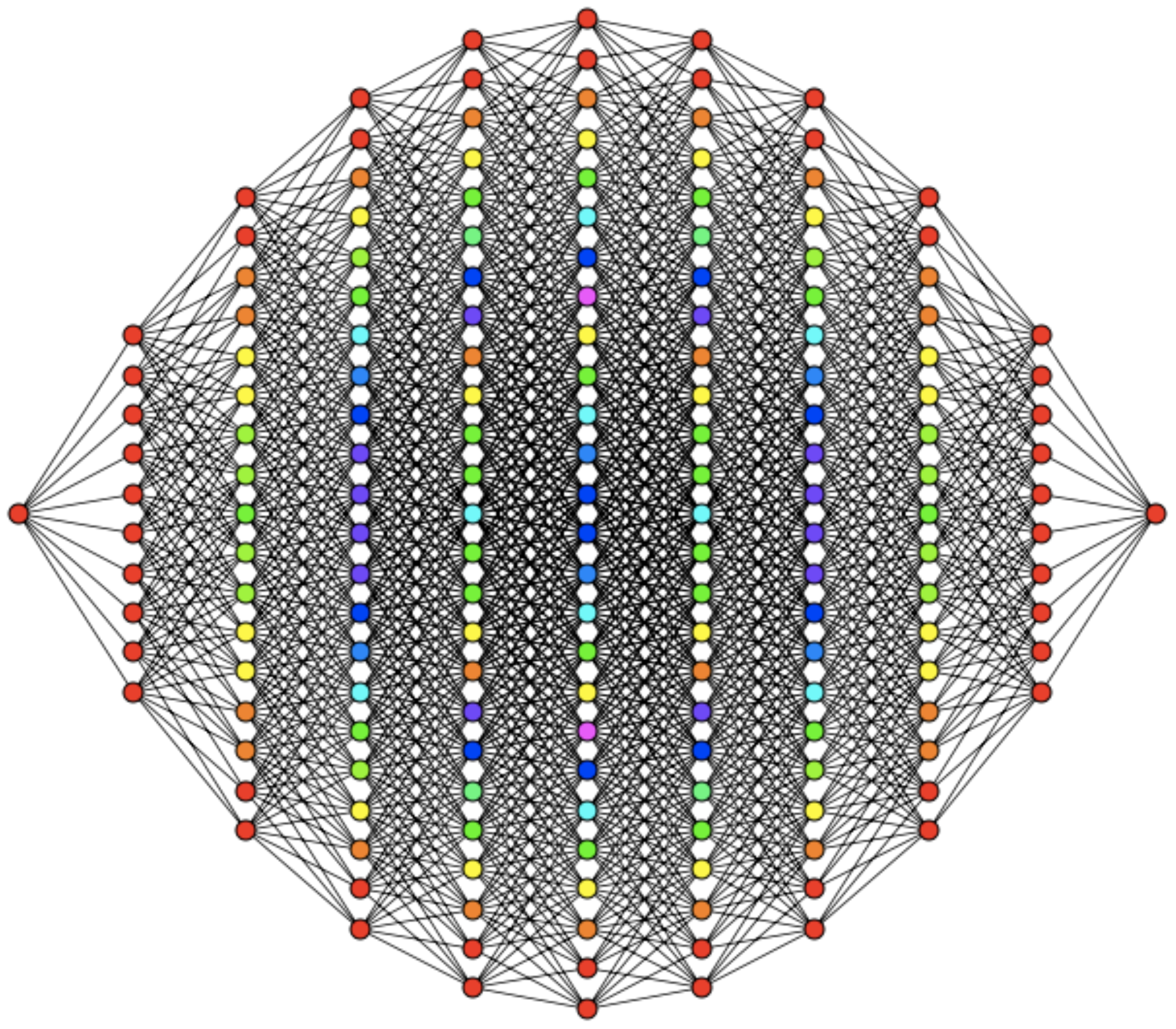}
\caption{The hypercube in $N= 10$ dimensions}
\label{diamond}
\end{center}
\end{figure}

Let now $\Pi_N$ be the set of paths of length $N$ from $\boldsymbol 0$ 
to $\boldsymbol 1$. Each $\pi  \in \Pi_N$ is of the form $\boldsymbol 0 = v_0, v_1, \dots, v_N = \boldsymbol 1$.  For any such paths, we may consider the sum of edge weights 
\beq
X_\pi \defi \sum_{(v_j, v_{j-1}) \in \pi} \xi_{v_{j-1}, v_j} 
\eeq
A quantity of interest is the minimal weight $m_N \defi \inf_{\pi \in \Pi_N} X_\pi$ 
when $N\to \infty$. This is a first passage percolation problem in the "mean field limit". 
To get a feeling of what is going on here, let us consider a reference path 
$\pi \in \Pi_N$ and $x\in \R_+$: by Markov inequality (union bound) we have
\beq \bea \label{diamond_one}
\PP\left[ m_N \leq x   \right] \leq (N!) \PP\left[ \sum_{(v_j, v_{j-1}) \in \pi} \xi_{v_{j-1}, v_j}  \leq x\right]
\eea \eeq
The sum of $N$ independent (mean one) exponentials $\xi_1, \dots, \xi_N$ is a well-studied 
object (one even has a closed form for its distribution - the Erlang distribution); in particular, the following asymptotics is easily checked
\beq
\PP\left[ \sum_{i=1}^N \xi_i \leq x \right] \sim e^{-x} \frac{x^N}{N!}
\eeq
Using this in \eqref{diamond_one}, we see that $\PP\left[ m_N \leq x \right] \to 0$ for $x< 1$. It turns out that this simple "REM-bound" is tight:
\begin{fact} 
\beq \label{aldous_fill_pemantle}
\lim_{N\to \infty} m_N = 1,
\eeq 
in probability. 
\end{fact}
The above \eqref{aldous_fill_pemantle} has been conjectured by Aldous \cite{aldous}, who also suggested that a good way to tackle the problem is by means of  {\it non-homogeneous dependent branching processes}. The conjecture has been settled 
 by Fill and Pemantle \cite{fill_pemantle}, although following a different route (a conditional second moment method with "variance reduction"). On the other hand, the multiscale refinement of the second moment method discussed above seems to be particularly suited to carry out Aldous' original strategy \cite{akz}.

\chapter{How to identify scales?}

\section{Musing on the Riemann $\zeta$-function}
Fyodorov, Keating and Hiary \cite{fyodorov_keating_hiary} have recently put forward the conjecture that the mechanism underlying the extremes of the 
hierarchical field might also be found in certain number-theoretical issues, such as the statistics of the extremes of the Riemann $\zeta$-function along the critical line. The latter is of course the function 
\beq
\zeta(s) = \sum_{n=1}^\infty \frac{1}{n^s} = \prod_{p\, \text{prime}} \left( 1- \frac{1}{p^s}\right)^{-1} \,.
\eeq
According the Riemann Hypothesis, the zeros of this function lie on the critical line $s= 1/2+ i t, t \in \R$. In general, many questions in the theory of the $\zeta$-function concern the distribution of values 
on the critical line. A theorem of Selberg \cite{selberg} states  that 
\beq
\lim_{T\to \infty} \frac{1}{T} \text{Leb}\left( T \leq t \leq 2 T: \alpha \leq \frac{\log \big|\zeta(1/2+i t) \big|}{\sqrt{\frac{1}{2} \log \log(\frac{t}{2\pi})}} \leq \be \right)
= \int_{\alpha}^\beta e^{-x/2} \frac{dx}{\sqrt{2\pi}} \,.
\eeq
In other words, 
\beq 
V_t(x) \defi - 2 \log \Big|\zeta(\frac{1}{2}+ i (t+ x))\Big|
\eeq
"behaves" as a Gaussian random field (indexed by $x$).
This random field is however {\it correlated}.  In fact, denoting by $\left< \cdot \right>$ the average over an interval $[t-h/2, t+h/2]$ such that 
$\frac{1}{\log t} \ll h \ll t$, it is informally demonstrated in \cite{fyodorov_keating} (see \cite{bourgade} 
for a rigorous proof) that
\beq
\lim_{t\to \infty} \left< V_t(x_1) V_t(x_2) \right> \approx \begin{cases}
- 2 \log |x_1-x_2|, &\text{for}\;  \frac{1}{\log t} \ll |x_1-x_2| \ll 1 \\
2 \log \log t, & \text{for}\;|x_1-x_2| \ll \frac{1}{\log t}
\end{cases}
\eeq
What is  important for the discussion here is the {\it logarithmic} form of the correlation. The simplest model which displays a logarithmic correlation structure is indeed the hierarchical field 
(upon identifiying the labels with their binary-expansions). Based on this insight, Fyodorov, Keating and Hiary \cite{fyodorov_keating_hiary} put forward some intriguing conjectures concerning the high values of the Riemann zeta-function. In particular, they consider 
\beq
\zeta_{\text{max}}(L, T) \defi \max_{0\leq x \leq L} V_T(x)
\eeq
and conjecture that, for $L \leq 2\pi$,
\beq
\zeta_{\text{max}}(L, T) \sim 2\log \left( \frac{L}{2\pi} \log \frac{T}{2\pi}\right) 
- \frac{3}{2} \log\log \left(  \frac{L}{2\pi} \log \frac{T}{2\pi}\right) 
\eeq
A similar formula but with the factor $3$ replaced by $1$ would hold if one assumes the 
random field to be independent: this is the hierarchical field vs. REM scenario. (Similar conjectures 
are also available for the characteristic polynomials of CUE random matrices, see  \cite{fyodorov_keating_hiary}.)\\

\section{On local projections}

It should be clear by now that the point of view discussed in the previous sections relies crucially on the identification of {\it scales} for the models at hand. Once these are identified, the (multiscale refinement of the) second moment method is well suited to address the question of the extremes. It goes without saying, different models come with varying degrees of technical difficulties, but certain models seem puzzling already at a conceptual level: where are the scales in the Riemann zeta-function? 
For the hierarchical field this question is easily answered: thanks to the in-built hierarchical structure,
one can even {\it visualize} the scales, and it is clear what is meant by, say, "larger scale".
On a structural level, the situation is also easy for the GFF, thanks to the Markovianity and the self-similarity of the field which lead to the representation \eqref{multiscale_full}. However, none of these properties are available in the case of the Riemann zeta-function, say. 
{\it Is there a general recipe which allows us to construct the scales, i.e. an underlying tree-like structure, from first principles?} This question should be taken with caution: identifying an underlying tree-structure is the foremost step in the implementation of the Parisi theory  \cite{parisi}. Despite the numerous advances in the field, see e.g. \cite{panchy, talagrand}, an understanding of this issue is yet nowhere in sight. On the other hand, models in the REM-class 
undergo what physicists refer to as the REM-freezing transition (a prominent feature is the lack of 
the so-called chaos in temperature). I will not go into any detail here, 
but I simply mention this to state the claim that for such models, a general recipe to identify scales might already be hiding in the approach discussed above. In fact, there is compelling evidence that the scales are related to certain {\it local projections}. To see this, 
let us go back to the GREM$(2)$, namely the model 
\beq \label{repr}
X_{\alpha} = X_{\alpha_1}^{(1)}+ X_{\alpha_1, \alpha_2}^{(2)}, 
\eeq
where $\alpha_1, \alpha_2 = 1 \dots 2^{N/2}$
and $X^{(1)}_\cdot \sim \mathcal N(0, a_1 N)$, $X_\cdot^{(2)} \sim \mathcal N(0, a_2 N)$, all independent. Again, we assume the normalization $a_1+a_2  = 1$.  The covariance of this Gaussian field naturally induces a metric on the configuration space 
$\Sigma_N \defi \{\alpha = (\alpha_1, \alpha_2): \alpha_1 = 1\dots 2^{N/2}, \alpha_2 = 1 \dots 2^{N/2}\}$: 
\beq 
d_N(\alpha, \alpha') \defi \sqrt{\mathbb{E} (X_\alpha - X_{\alpha'})^2} = \sqrt{N - q_N(\alpha, \alpha')}, 
\eeq
where $q_N(\alpha, \alpha') = \mathbb{E} X_\alpha X_{\alpha'}$ is the overlap of two configurations. Since 
the configuration space is endowed with a geometry  (eventually induced by the covariance structure), we may introduce the concept of {\it neighborhood} of a configuration. One possible definition is 
\beq
B_r(\alpha) \defi \{\alpha': q_N(\alpha, \alpha') \leq r N \} \,.
\eeq
for $\alpha \in \Sigma_N$, and $r\geq 0$.

As we have seen, a good way to tackle the extremes of GREM$(2)$  is to keep track of first {\it and} second level, which we achieved by introducing the counting 
random variable 
\beq 
\mathcal N_N(\la_1, \la_2) = \sharp \left\{ \alpha = (\alpha_1, \alpha_2): X_{\alpha_1}^{(1)} \geq 
\la_1 N, X_{\alpha_1, \alpha_2}^{(2)} \geq \la_2 N \right\}\,.
\eeq
Of course, knowing the specific representation \eqref{repr} is of help to identify first and second level. In many interesting models, however, the equivalent of \eqref{repr} is lacking, so it would 
be of interest to have a procedure which generates the levels from "first principles". Here is a way which  unravels the underlying tree-like structure using only the covariance of the field. 

Consider a configuration $\alpha = (\alpha_1, \alpha_2)$. The idea is to decompose telescopically 
\beq 
X_{\alpha} = \mathbb{E}\left[ X_{\alpha} \mid \mathcal F_r(\alpha) \right] + \left( X_{\alpha} - \mathbb{E}\left[ X_{\alpha} \mid \mathcal F_r(\alpha) \right]  \right)\,,
\eeq
where $\mathcal F_r(\alpha)$ is the sigma-field generated by all random variables which 
are in the $r$-neighborhood of $\alpha$. Due to the simple covariance structure of the 
GREM(2), there are only few sensible choices for the "radius", namely $r = a_1$ or $1$. In the case 
$r=1$ the neighborhood consists of the whole configuration space, so the only non-trivial choice is, in fact, $r= a_1$. Let us take a closer look at the conditional expectation: in this case
$\mathcal F_{a_1}(\alpha)$ is the $\sigma$-field generated by all random variables for which either
$\alpha_1' \neq \alpha_1$, or $\alpha_1' = \alpha_1$ but $\alpha_2' \neq \alpha_2$.
 But the random variables $X_\alpha$ and $X_{\alpha'}$ are independent as soon as $\alpha'_1 \neq \alpha_1$, so we may completely dismiss the collection$\{X_{\alpha'}, \alpha'_1 \neq \alpha_1\}$, i.e.
\beq
\mathbb{E}\left[ X_{\alpha} \mid \mathcal F_{a_1}(\alpha) \right] = 
\mathbb{E}\left[ X_{\alpha} \mid X_{\alpha'}: \alpha_1' = \alpha_1, \alpha_2' \neq \alpha_2\right].
\eeq
In other words, we are conditioning on those $\alpha'$ which share the 
first index with the reference configuration.  Since all involved random variables are Gaussian, conditional expectation are nothing but linear combinations of the random variables upon which one 
conditions: skipping the tedious calculations, one gets
\beq \bea
\mathbb{E}\left[ X_\alpha \mid \mathcal F_{a_1}(\alpha_1) \right]  & = \frac{1}{1 + 2^{-N/2} (1-a_1)}  X_{\alpha_1}^{(1)} +  \frac{2^{-N/2}}{a_1 + (1-a_1) 2^{-N/2}} \sum_{\tau = 1, \tau\neq \alpha_2}^{2^{N/2}} X_{\alpha_1, \tau}^{(2)}.
\eea \eeq
By the law of large numbers, the second term above is, in the large $N$ limit,  {\it exponentially} small, $\PP-$almost surely. Furthermore 
\beq
\frac{1}{1 + 2^{-N/2} (1-a_1)} = 1-O\left(2^{-N/2}\right),
\eeq
hence
\beq
\mathbb{E}\left[ X_{\alpha} \mid \mathcal F_{a_1}(\alpha)\right] = X_{1}^{(1)} + 
\Omega_N.
\eeq
where $\Omega_N$ is an exponentially small term, and consequently
\beq
X_{\alpha} - \mathbb{E}\left[ X_{\alpha} \mid \mathcal F_{a_1}(\alpha) \right]  = X_{\alpha_1, \alpha_2}^{(2)} + \Omega_N. 
\eeq
By locally projecting the field onto a neighborhood measured w.r.t. the metric 
induced by the covariance, we have thus constructed first and second level of the tree (up to errors which are completely irrelevant in the large $N$-limit). \\

Since local projections work smoothly in case of the GREM \footnote{One may follow analogous steps in order to identify the scales in a GREM$(K)$, for generic $K$. In this case one simply decomposes telescopically into a sum of $K$ terms ensuing from local projection on larger and larger neighborhoods. }, it is natural to test the method on the GFF. We will see that, indeed, it allows us to generate from first principles the hierarchical decomposition \eqref{multiscale_full}. To see this,  consider the GFF $\{X_\eta^N, \eta \in V_N\},$ where $V_N\subset \mathbb{Z}^2$ is a box of size $N$. In this case the overlap of two configurations (sites) is given by
\beq \label{gff_overlap}
q_N(\eta, \eta') \defi \mathbb{E} X_{\eta}^N X_{\eta'}^N = \frac{2}{\pi} \left( \log N - \log \| \eta- \eta' \|_2\right)
\eeq
where $\|\cdot \|_2$ is the Euclidean distance (at least for $\eta, \eta'$ far enough from the boundary of the box). Let $q\in \R_+$. Implementing the approach through local projections we get the 
following telescopic decomposition 
\beq
X_\eta^N = \mathbb{E}\left[ X_{\eta}^N \mid X_{\eta'}^N: \; q_N(\eta, \eta') \leq q \log N\right]
+ \left( X_{\eta}^N - \mathbb{E}\left[ X_{\eta}^N\mid X_{\eta'}^N: \; q_N(\eta, \eta') \leq q \log N\right] \right)
\eeq 
What is the "neighborhood" in case of the GFF? By \eqref{gff_overlap}, it holds
\beq 
\{\eta' \in V_N: \; q_N(\eta, \eta') \leq q \log N \} = \{\eta' \in V_N: \;  \| \eta'- \eta \|_2 \geq N^{(1-\pi q/2)} \},
\eeq
namely the {\it complement} of the Euclidean ball of radius $ N^{(1-\pi q/2)}$ centered in $\eta$. 
Denoting by $C_\eta(q)$ such Euclidean ball, by Markovianity of the GFF we see that conditioning the field upon the complement $C_{\eta}(q)^\texttt{c}$ coincides with conditioning the field on the {\it boundary} of $C_\eta(q)$. Local projections thus lead to the decomposition 
\beq \label{simple_wonder}
X_\eta^N = \mathbb{E}\left[ X_{\eta}^N \mid X_{\eta'}^N: \; \eta' \in \partial C_\eta(q) \right]
+ \left(X_\eta^N- \mathbb{E}\left[ X_{\eta}^N \mid X_{\eta'}^N: \; \eta' \in \partial C_\eta(q) \right] \right) ,
\eeq
for a $q$ which may be chosen as we wish. We hardly expect any difference between conditioning upon the sites which are on a (Euclidean) square or on a (Euclidean) ball, so the above representation can be safely identified with \eqref{first tel}.  Iterating the procedure for the second term in the telescopic 
decomposition \eqref{simple_wonder} (this step is particularly easy here, thanks to the self-similarity of the field) one then immediately obtains \eqref{multiscale_full}. \\

Since the procedure identifies from first principles the tree-structure underlying the GREM or that of the GFF, one can only wonder if local projections capture some fundamental aspects lying underneath the surface of models in the REM-class. A natural "playground" would be of course the case of the Riemann $\zeta$-function, or the characteristic polynomials of CUE random matrices \cite{fyodorov_keating_hiary, fyodorov_keating}.  On a more spin glass side,  it would be interesting to see if the local projections allow to identify the scales behind the Parisi landscape in finite-dimensional Euclidean spaces   which have been introduced in \cite{bouchaud_fyodorov}. These models have been rigorously analyzed in \cite{klimovsky}  by means of Guerra's interpolation scheme \cite{guerra} 
and by \cite{az} by means of the Ghirlanda-Guerra identities \cite{ghirlanda_guerra}, which are to these days among the most powerful yet mysterious tools in spin glasses. It would be interesting to have a complementary, more transparent approach. \\

Let me conclude with a caveat. The local projections discussed above should not be understood as a "frontal attack" to models in the REM-class. In fact, the procedure involves conditional expectations: 
these are particularly easy to handle if the underlying random variables are Gaussian, but they quickly become demanding/untractable otherwise. This is however a technical difficulty, as opposed to the more structural quest of bringing to the surface "hidden geometries". It is the emerging geometrical picture, and perhaps only this, which might be a good starting point for the analysis.

\bigskip
{N. KISTLER, Institut f\"ur Mathematik, Goethe-Universit\"at Frankfurt, DE-60054 Frankfurt am Main,
 \href{mailto:kistler@math.uni-frankfurt.de}{kistler@math.uni-frankfurt.de}}
\vfill \hfill \includegraphics[scale=0.06]{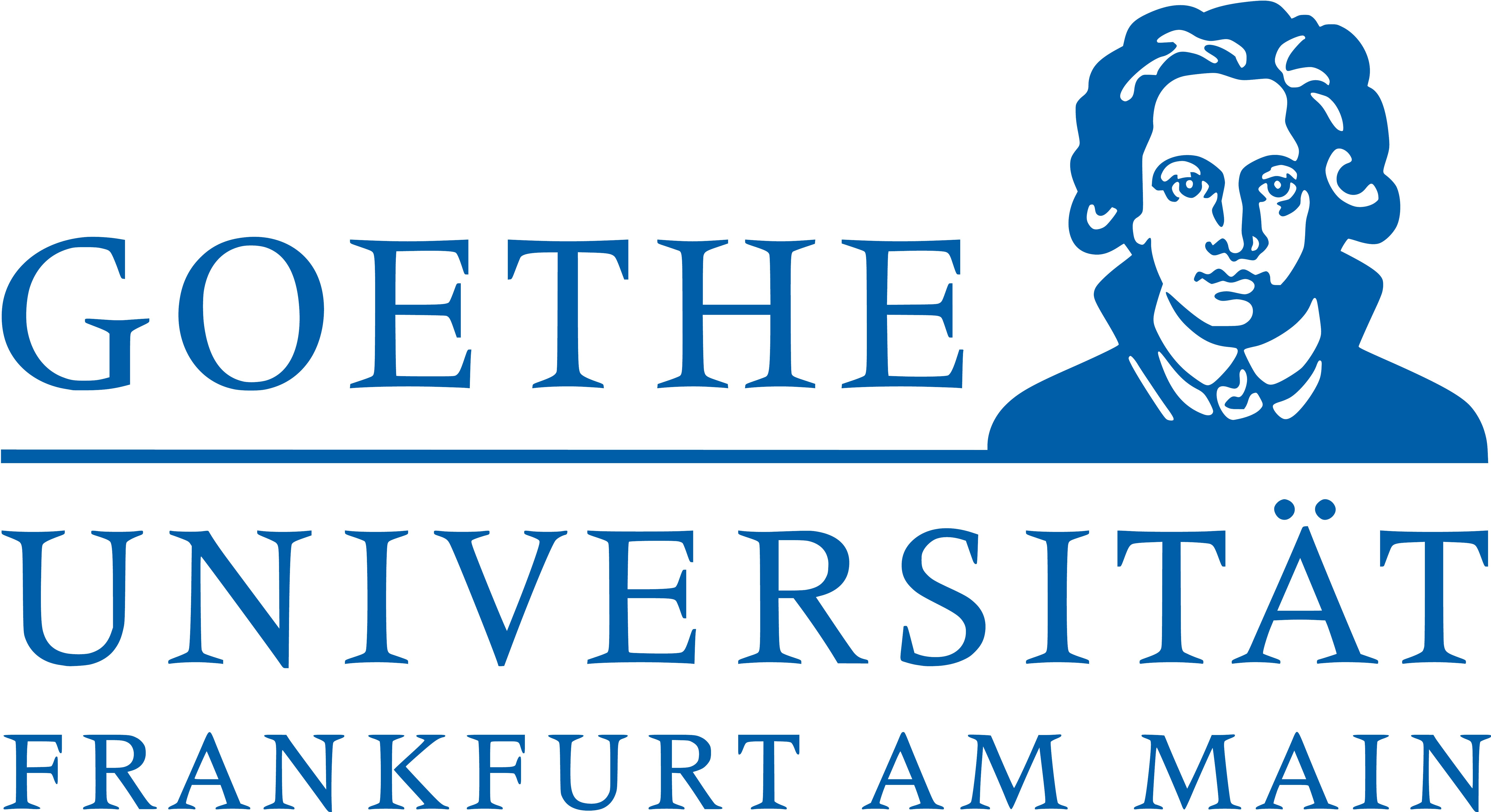}



\end{document}